\documentclass[a4paper,12pt,openany]{article}
\usepackage[latin1]{inputenc}
\usepackage{calc}
\usepackage[all]{xy}
\usepackage[centertags]{amsmath}
\usepackage[section]{placeins}
\usepackage{latexsym}
\usepackage{amsfonts}
\usepackage{cases}
\usepackage{makeidx}
\usepackage{eurosym}
\usepackage{wasysym}
\usepackage{array}
\reversemarginpar \DeclareMathAlphabet{\mathpzc}{OT1}{pzc}{m}{it}
\usepackage{amssymb}
\usepackage{amsthm}
\usepackage{multicol}
\usepackage{rotating}
\usepackage{multirow}
\usepackage[T1]{fontenc}
\usepackage{lmodern}
\usepackage{color}
\usepackage[colorlinks=true, linkcolor=blue
]{hyperref}
\usepackage{fancyhdr}
\usepackage [dvips]{epsfig}

\usepackage{newlfont}

\newtheorem{prop}{\textbf{Proposition}}
\theoremstyle{plain}
\newtheorem{rmk}{\textbf{Remark}}
\theoremstyle{plain}
\newtheorem{theo}{\textbf{Theorem}}
\newtheorem{defi}{\textbf{Definition}}
\theoremstyle{plain}

\newtheorem{lem}{\textbf{Lemma}}
\theoremstyle{plain}
\theoremstyle{plain}

\theoremstyle{plain}
\newenvironment{prev}[1][Proof]{\textbf{#1.} }{\ \rule{0.5em}{0.5em}}
\theoremstyle{plain}
\usepackage{comment}

\usepackage[latin1]{inputenc}

\usepackage{graphicx}
\usepackage{t1enc}

\usepackage{listings,  color}
\definecolor{gray}{RGB}{234, 244, 243}
\title{\textbf{Existence and uniqueness for the solutions of non-autonomous stochastic differential algebraic equations with locally Lipschitz coefficients} \vspace*{.55cm}\\
	\author{  Oana Silvia Serea \footnote{oana.silvia.serea@hvl.no: Western Norway  University of Applied Sciences, Norway and University of Perpignan, France},  Antoine Tambue \footnote{antoine.tambue@hvl.no: Western Norway University of Applied Sciences, Norway and AIMS-South-Africa}, and  Guy Tsafack \footnote{Guy.Tsafack@hvl.no : Western Norway University of Applied Sciences, Norway},}}
\begin{document}
	\baselineskip 6mm
	\maketitle
	\noindent

	\section*{Abstract} In this paper, we study the well-posedness and regularity of  non-autonomous stochastic differential algebraic equations (SDAEs) with nonlinear, locally Lipschitz and monotone \eqref{equa3} coefficients of the form  \eqref{equa1}.
	The main difficulty is the fact that the operator $A(\cdot)$ is non-autonomous, i.~e. depends on t and the matrix $A(t)$ is singular for all $t\in \left[0,T\right]$. Our interest is in SDAE of index-1. This means that in order to solve the problem, we can transform the initial SDAEs into an ordinary stochastic differential equation with algebraic constraints.  Under appropriate hypothesizes, the main result establishes the existence and uniqueness of the solution in $\mathcal{M}^p(\left[0, T\right], \mathbb{R}^n)$, $p\geq 2$, $p\in \mathbb{N}$. Several strong estimations and regularity results are also provided. Note that, in this paper, we use various techniques such as  It\^o's lemma, Burkholder-Davis-Gundy inequality, and Young inequality.\\

	\textbf{Key words:}   Non-autonomous stochastic differential-algebraic equations, nonlinear and locally Lipschitz coefficients, index-1 differential-algebraic equations.

	\section{Introduction}
	We are interested in the well-posedness of the following stochastic differential algebraic equation:
	\begin{equation}\label{equa1}
		A(t)dX(t)=f(t,X(t))dt + g(t,X(t))dW(t), \quad t\in \left[0,T\right] ,\quad X(0)=X_0.
	\end{equation}
	Here $T>0,\quad T\ne \infty$, and  A : $ \left[0,T\right] \to \mathcal{M}_{n\times n}(\mathbb{R})$ is continuously differentiable. In addition  $ A(t)\in  \mathcal{M}_{n\times n}(\mathbb{R})$ is a singular matrix for $ t\in \left[0,T \right]$.
	The functions  $f:\left[0,T\right]\times\mathbb{R}^n  \to \mathbb{R}^n  $ (drift) and $g:\left[0,T\right]\times\mathbb{R}^n \to \mathbb{R}^{n\times m} $ (diffusion) are Borel measurable. 
	The process $W(\cdot)$ is a m-dimensional  Wiener process defined on the probability space $(\Omega, \mathcal{F},\mathbb{P})$ with the natural filtration  $\left( \mathcal{F}_t\right)_{t\geq 0} $. 
	The unknown function  $X(\cdot)$ is a vector-valued stochastic process of dimension  $n$ and depends both on time $ t\in \left[0,T \right] $ and the sample $\omega \in \Omega$. To ease the notation, the argument $\omega$ is omitted. 
	\\
	
	Equations of type \eqref{equa1} arise in a large range of scientific (mathematics, physics, biology, chemistry, etc.), engineering, financial, and electrical applications \cite{Contandtankov2004, Gikhman1972,  Klebaner2004, Renate2003} and are currently studied by many researchers. 
	These equations are called stochastic differential algebraic equations (SDAEs). Note that SDAEs are a combination of algebraic constraints (AEs) and stochastic differential equations (SDEs). Detailed introduction of deterministic differential algebraic equations under the global Lipchitz conditions on the drift coefficient can be found in \cite{Campbell1995, Griepentrog1986, Wanner1991}, while the basic theory for stochastic differential equations under the global Lipchitz conditions on the drift and the diffusion coefficients can be found in \cite{Arnold1974,Friedman2010, Lawrence2012, Gikhman1972, Kloeden1992}. Many practical applications do not satisfy the global Lipchitz condition \cite{Burrage2003, Hutzenthaler2020,lord2022, Lewis2009}. In \cite{Guixin2011, Mao2007, Yong2015}, the existence and uniqueness under the locally Lipschitz condition on the coefficients are investigated for SDEs. However, SDEs are not enough to describe many phenomena. Indeed, in \cite{Demir2012, Thorsten2008, Oliver1999, Maten1999}, it has been proven that SDAEs of type  \eqref{equa1} are needed for modeling electrical circuits. In \cite{Oliver1999}, the authors studied the existence and uniqueness of solutions for SDAEs with linear drift coefficients. The general theory of SDAEs was investigated in \cite{Renate2003}, where the authors proved the existence and uniqueness of the solution under the global Lipschitz conditions on the coefficients. It is obvious that if matrix $A(t)$ in \eqref{equa1}  is a non-singular matrix, the SDAEs become SDEs. To the best of our knowledge, only index 1-SDAEs with a constant matrix $A(t)=A$ have been studied up to now. As mentioned before, in several practical problems, the matrix $A(t)$ is not a constant matrix \cite{Campbell1995,  Kunkel1995, Marz1995}; such equations are called non-autonomous SDAEs. In \cite{Nguyen2010, Oliver1999}, the authors proved the well-posedness of the non-autonomous linear SDAEs.  Unfortunately, real-world problems are nonlinear and linear equations describe very few problems.
	Therefore the study of fully nonlinear SDAEs of type \eqref{equa1}  is an interesting research topic.
	To the best of our knowledge, the existence and uniqueness of the solution for non-autonomous nonlinear SDAEs of type \eqref{equa1}  under local Lipschitz conditions is still an open problem. Due to the non-linearity of the drift and the noise coefficients, 
	the singularity of the matrix $A(\cdot)$ and the non-autonomous of SDAEs of type \eqref{equa1},  the study becomes more challenging, and there is no hope to find analytical solutions.
	Therefore numerical algorithms are the only tools to provide accurate approximations. 
	
	The design and the rigorous mathematical study of convergence (weak or strong) of a numerical algorithm to approximate  SDAEs of type \eqref{equa1}  depends on the regularity of the solution.
	The regularity of solutions for SDEs is well understood in the literature see \cite{Jiang2018, lord2022, Tambue2019}. However, the regularity of the solution of SDAEs of type \eqref{equa1}, particularly when the coefficients of the equations satisfy only the local Lipschitz condition and $A(\cdot) $ depends on $t$ is still open to the best of our knowledge. 
	
	This paper aims to prove the well-posedness and regularity of the solution for nonlinear non-autonomous SDAEs under local Lipschitz conditions and the monotone conditions \eqref{equa3}.  In order to achieve our goal, we transform the initial SDAEs into ordinary SDEs with algebraic constraints (AEs).  Although we follow a similar approach to \cite{Renate2003}, where constant matrix  $A(t)=A$ is constant, we face many challenges. In order to tackle them, we use It\^o 's lemma to transform the equation with the variable $X$ into a new ordinary SDE with a new variable $u$. This is done using matrix transformations and the implicit function theorem. We show that the new SDE called inherent regular SDE satisfies the locally Lipschitz condition and the monotone conditions \eqref{equa3} on the new drift and the new diffusion coefficients.   The existence and uniqueness  of  $u$ is then ensured \cite{Guixin2011, Mao2007, Yong2015}.
	Furthermore, we study the regularity of the solution of SDAE under relaxed assumptions.
	
	The paper is organized as follows.  Section 2 presents several basic notions and assumptions. Section 3 contains some classical existence, uniqueness, and regularity results for the solutions of global SDEs. In section 4, we prove the well-posedness of the solution of the SDAE \eqref{equa1} by transforming the initial equation \eqref{equa1} into an ordinary SDEs with algebraic constraints (AE). We end this section by showing that the solution $X$ of the equation \eqref{equa1}   satisfies, in addition,  the estimation $\mathbb{E}(\sup_{t\in \left[0, T\right]}\left|X(t)\right|^p)\leq C$ . Here $C> 0$ is a constant and $p\geq 2$. Finally, we give an example to illustrate our approach.
	
	\section{Notations and assumptions}
	Throughout this work,  $(\Omega, \mathcal{F},\mathbb{P})$ denotes a complete probability space with the natural 
	filtration $\left( \mathcal{F}_t\right)_{t\geq 0} $. We define,  $\left\|x \right\|^2=\sum_{i=1}^{n}\left| x_i\right|^2  $ , for any vector $x\in \mathbb{R}^n$ and   $\left| B \right|^2_F=\sum_{i=1}^{n}\sum_{j=1}^{m}\left|b_{i,j}\right| ^2 $ is the Frobenius norm, for any matrix $B=(b_{i,j})_{i,j=1}^{n, m}$.  We denote also by $\left\|\cdot\right\|_{\infty}$ the supremum norm for continuous functions. Moreover, for $A:t\mapsto  A(t),~t\in \left[ 0,T\right]$, we have  $\left\| A\right\|_{\infty}:=\max_{t\in \left[ 0,T\right]}\left| A(t)\right| _F$. Recall that, $a\lor b=\max(a,b)$ is the maximum between $a$  and $ b$,  and  $a\land b=\min(a,b)$ is the minimum between $a$ and $b$, for $a$, $b$ $\in \mathbb{R}$ .	
	
	We continue with the following hypotheses necessary for obtaining our first results. More precisely, we assume that:\\
	\textbf{Assumption 1.}\\
	\textbf{(A1.1)} The functions $f(\cdot, \cdot)$ and $g(\cdot, \cdot)$ satisfy the following monotone  condition  i.e. there exists $k>0$ such that:
	\begin{equation}\label{equa3}
		\left\langle (A^-(t)A(t)X)^T,A(t)^-f(t,X)\right\rangle +\dfrac{p-1}{2}\left|A^-(t)g(t,X) \right|^2_F \leq k(1+\left\|X \right\|^2 ),~X \in \mathbb{R}^n,~t \in \left[0,T\right].
	\end{equation}
	Here $A^-(\cdot)$ is the so-called pseudo-inverse of $A(\cdot)$ and $p\geq 2$.\\
	Note that in the case of normal SDEs, this condition is replaced by the following conditions:
	Suppose that there exists a positive constant $k>0$ such that:
	\begin{equation}\label{eqat}
		\left\langle X^T,f(t,X)\right\rangle +\dfrac{p-1}{2}\left|g(t,X) \right|^2_F \leq k(1+\left\|X\right\|^2 ), X\in \mathbb{R}^n, ~t\in \left[0,T\right].
	\end{equation}
	So, we will prove that the SDE obtained after the transformation of the initial SDAE satisfies this condition \eqref{eqat}. 
	
	\textbf{(A1.2)}	 	The functions $f(\cdot,\cdot)$ and $g(\cdot,\cdot)$ are locally Lipschitz with respect to $X$ i.e. for any $q > 0$,   there exists $L_q > 0$, such that:
	\begin{equation}\label{equa4}
		\left\| f(t,X)-f(t,Y)\right\|\lor	\left| g(t,X)-g(t,Y)\right|_F \leq L_q \left\|Y- X\right\|,~t\in \left[0,T\right],
	\end{equation}
	and  $ X, Y\in \mathbb{R}^n$ with $ \left\|X \right\|\lor \left\|Y \right\|<q $ and $t \in \left[0,T\right]$.\\
	\textbf{Assumption 2.} 
	In the following, we consider that the function $g(\cdot,\cdot)$ satisfies the super linear condition. This condition will be used only for the regularity of results. More precisely; we suppose that:  there exist $ c, r_2 >0$
	such that:
	\begin{equation*}
		\textbf{(A2.1)}~~~~~~~~~~~~~~~~~~~~~~~~~~~~~~~~~~	\left|g(t,X) \right|_F^2\leq c(1+\left\|X \right\|^{r_2} ), ~ X\in \mathbb{R}^n, ~t\in \left[0,T\right].
	\end{equation*} 
	
	Finally, our last hypothesis is: \\
	\textbf{(A2.2)} The matrix function , $P:\left[0,T\right]\to \mathcal{M}_{n\times n }(\mathbb{R})$  such that $P(t)=A^-(t)A(t)$ is of class $\mathcal{C}^1$ with respect to $t\in \left[0,T\right]$.\\
	
	The next section presents some background results on stochastic differential equations and several results concerning singular matrices.
	\section{Some classical results on stochastic differential equations and singular matrices }
	In this section, we present some important results concerning the existence and uniqueness of the solution for  SDEs under the local Lipschitz conditions. Moreover, we present several useful results for the study of differential-algebraic equations.
	\subsection{Stochastic differential equations}
	\begin{defi}\cite{Arnold1974,Gikhman1972,Kloeden1992}
		A stochastic differential equation is an equation of the form:
		\begin{equation}\label{key1}
			dX(t)=f(t,X(t))dt+g(t,X(t))dW(t) ,\qquad X(0)=X_0,~t\in \left[ 0,T\right].
		\end{equation} 
		Equivalently, we have the integral form:
		\begin{equation}\label{key2}
			X(t)-X_0=\int_{0}^{t}f(s,X(s))ds+\int_{0}^{t}g(s,X(s))dW(s), ~t\in \left[0,T \right] .
		\end{equation}
		We continue with the definition of a solution of \eqref{key1} or  \eqref{key2}.
	\end{defi}
	\begin{defi}\cite{Thorstenbook2008}\label{def0}
		A strong solution of \eqref{key1} or  \eqref{key2} is a process $X(\cdot) = (X(t))_{t\in \left[0,T \right] }$ with continuous sample paths that respects the following conditions:\\
		i)  $X(\cdot)$ is adapted to the filtration $\left\lbrace \mathcal{F}_t\right\rbrace_{t\in \left[0,T\right] } $,\\
		ii) $\int_{0}^t|f_i(s,X(s))|ds< \infty$ a.s, for all i=1,...,n, $t\in \left[0,T\right]$.,\\
		iii) $\int_{0}^tg_{ij}^2(s,X(s))dw_j(s)< \infty$ a.s, for all i=1,...,n; j=1,...,m, $t\in \left[0,T\right]$,\\
		iv)  $X(\cdot) = (X(t))_{t\in \left[0,T \right] }$ satisfies the equation  \eqref{key1} or \eqref{key2}.
	\end{defi}
	Another important tool in stochastic analysis is the so-called It\^o process.
	
	\begin{defi}\cite{Mao2007}\label{def00}
		A one-dimensional It\^o process (or stochastic integral) is a continuous adapted process $(X(t))_{t\in \left[0,T \right] }$ of the form:
		
		\begin{equation}\label{ito}
			X(t)=X_0+\int_{0}^{t}f(s,X(s))ds+\int_{0}^{t}g(s,X(s))dW(s).
		\end{equation}
		Here  $f\in \mathbb{L}(\left[0,T \right],\mathbb{R} )$ and $g\in \mathbb{L}^2(\left[0,T \right],\mathbb{R} )$. We assert that $X(\cdot)$ possesses a stochastic differential $dX(\cdot)$ as described by equation \eqref{key1}.
	\end{defi}
	We would like to recall a fundamental result in stochastic theory, which holds significant importance.
	\begin{prop}[ \textbf{It\^o lemma}]\label{propito} \cite{mukam2015}
		Let $(\Omega,\mathcal{F}, \mathbb{P})$ be a complete probability space, $\{W(t)\}_{t\in  \left[0,T \right]}$ a one-dimensional Brownian motion and $K: \left[0,T \right] \times\mathbb{R}\to \mathbb{R}$ such that $K$ is once differentiable with respect  to the first variable $t$ and twice differentiable with respect to the second variable $X$. If $\{X(t)\}_{t\in  \left[0,T \right]}$ is an It\^o  process, then $\{K(t, X(t))\}_{t\in  \left[0,T \right]}$ is also an It\^o process and we have:
		\begin{align}\label{ito2}
			K(t,X(t))=&K(0,X(0)+\int_{0}^{t}\frac{\partial K}{\partial t}(s,X(s))ds+\int_{0}^{t}\frac{\partial K}{\partial X}(s,X(s))f(s,X(s))ds\nonumber\\
			&+\int_{0}^{t}\frac{\partial K}{\partial X}(s,X(s))g(s,X(s))dW(s)+\frac{1}{2} \int_{0}^{t}\frac{\partial^2 K}{\partial X^2}(s,X(s))g(s,X(s))^2ds, \nonumber\\              &      ~t\in\left[0,T \right].
		\end{align}
		Equivalently we have the differential form
		\begin{align*}
			dK(t,X(t))& =\frac{\partial K}{\partial t}(t,X(t))dt+\frac{\partial K}{\partial X}(t,X(t))f(t,X(t))dt+\frac{\partial K}{\partial X}(t,X(t))g(t,X(t))dW(t)\\
			& +\frac{1}{2}\frac{\partial^2 K}{\partial X^2}(t,X(t))g(t,X(t))^2dt, ~t\in \left[0,T \right].
		\end{align*}
	\end{prop} 
	\begin{prev}
		See \cite[Theorem 6.2 page 32]{Mao2007} 
	\end{prev}\\
	Let us recall that if the coefficients of equation \eqref{key1} satisfy condition (A1.2) and condition \eqref{eqat} (with $p=2$) as stated in Assumption 1, then according to \cite[Theorem 3.5, page 58]{Mao2007}, the solution $X(\cdot)$ of equation \eqref{key1} exists and is unique.
	
	\begin{theo}\label{theoexi}
		Assume that the monotone condition \eqref{eqat} (with p=2) and the local Lipschitz condition  \eqref{equa4} hold and $\mathbb{E}(\left\|X_0\right|^{2})< \infty$. Then there exists  an unique solution $X(\cdot)\in \mathcal{M}^2([0,T];\mathbb{R}^n)$\footnote{ $\mathcal{M}^2\left( \left[0,T \right];\mathbb{R}^n \right) $ is the family of processes $\left\lbrace X(t) \right\rbrace_{0\leq t\leq T} $ in $\mathbb{L}^2\left(\left[0,T \right]; ~\mathbb{R}^n \right) $ such that $\mathbb{E}\left(\int_{0}^{T}\left|X(t) \right|^2  dt\right)<\infty$  .}  of the equation  \eqref{key1} or  \eqref{key2}. Moreover, we have
		\begin{equation}\label{eq}
			\mathbb{E}\left( \sup_{t\in\left[0,T \right] }\left|X(t) \right|^{2}\right)  \leq C,
		\end{equation} where $C>0$ is a constant depending on $X_0$.
	\end{theo}
	\begin{proof}
		For the proof of the existence and uniqueness of the solution, see  \cite[Theorem 3.5 page 58]{Mao2007}. For the estimation  \eqref{eq}, we can use the same approach as in \cite[ Lemma 3.2 page 51]{Mao2007}.
	\end{proof}
	
	Moreover, we have the following result about the regularity of the solution $X(\cdot)$:
	\begin{theo}\label{theoexi1}
		Let $p\geq2$ and $X_0\in \mathcal{M}^p(\mathbb{R}^n)$. Assume that Assumption 1 and Assumption 2 hold. Then the solution $X(\cdot)$ of the equation \eqref{key1} or  \eqref{key2} satisfies: 
		\begin{equation}
			\mathbb{E}\left( \sup_{t\in\left[0,T \right] }\left\|X(t) \right\|^{p-r_2+2} \right) \leq C.
		\end{equation}
		Here $C>0$ and $r_2$   are constants numbers and $r_2<p$ is given by Assumption 2.
	\end{theo}
	\begin{proof}
		For the proof see \cite[Lemma 4.2]{mao2015truncated}.
	\end{proof}
	We continue presenting some important results concerning the singular matrices.
	\subsection{Some results on singular matrices}
	We start by introducing the notion of a pseudo-inverse of a matrix and some consequences of this definition.
	Pseudo inverse matrix is a generalization of the inverse matrix. Indeed, if B is a non-singular matrix, then the pseudo inverse coincides with the inverse of the matrix B. However, if B is singular, the pseudo inverse is defined as follows:
	\begin{defi}\label{def2}  Let $m,n\in\mathbb{N}^*$ with $m,n>0.$ A pseudo inverse of  $B\in \mathcal{M}_{m\times n}(\mathbb{R})$ is defined by a matrix $ B^-\in\mathcal{M}_{n\times m}(\mathbb{R})$ satisfying the following criteria:\\
		i)  $BB^-B=B$, \\
		ii) $B^-BB^-=B^-$,\\
		Moreover the matrix $ B^-$ is unique if the following properties satisfied:\\
		iii) $(BB^-)^T=BB^-$,\\
		iv)  $(B^-B)^T=B^-B$.\\
		Here $B^T$ is the transposed of the matrix $B$.
		For more information on the pseudo-inverse matrices, see, for example, \cite{San2019} and \cite{James2007}.
	\end{defi}
	\begin{prop}\label{prop1}
		Let $B\in \mathcal{M}_{n\times n}(\mathbb{R})$ be a matrix and  $B: \mathbb{R}^n\to \mathbb{R}^n$ the associated linear operator. Then there exists a  projector matrix $Q$ onto $Ker(B)$ such that $Im(Q)=Ker(B)$. Moreover, we can find another projector matrix $R$ along $Im(B)$ such that $RB=0$.
	\end{prop}
	\begin{prev}
		The proof of the Proposition \ref{prop1} uses the well-known formula  $$KerB\bigoplus ImB=\mathbb{R}^n.$$
		Here $\bigoplus$ is the direct sum between two sub-spaces (see \cite{rakic2018note}) 
	\end{prev}\\
	\begin{prop}\label{prop2}
		Let $B \in \mathcal{M}_{n\times n}(\mathbb{R})$ a singular matrix, $Q \in \mathcal{M}_{n\times n}(\mathbb{R})$ is a projector given by Proposition \ref{prop1} such that $Im(Q)=Ker(B)$. Then there exists a suitable non-singular matrix $D \in \mathcal{M}_{n\times n}(\mathbb{R})$ such that $DB=I-Q=P$\footnote{Observe that $P=I-Q$ is also a projector.}.
	\end{prop}
	\begin{proof}
		The proof of Proposition \ref{prop2}  is proposed  at  Appendix \ref{Appendix}.
	\end{proof}
	
	Note that at this point, we have the most important material to start the main part of this article. 
	\section{Non-autonomous stochastic differential-algebraic equations: Main results }
	In the first step, we present results concerning the well-posedness of non-autonomous SDAEs under locally Lipschitz conditions. Secondly, we address the regularity of the solution.
	\subsection{Existence and uniqueness results for non-autonomous SDAEs}

	Recall that, under suitable hypothesis, the non-autonomous stochastic differential algebraic equation (SDAE) \eqref{equa1} can be split into two different equations:  algebraic equation  (AE) and stochastic differential equation (SDE). The AE is called the constraints of the SDAE, and the solution of AE  is called algebraic variables. In this paper, we consider SDAEs of index 1.
	
	\begin{defi}\label{def1}
		The SDAE (\ref{equa1}) is said to be of index-1  if the noise sources do not appear in the constraints, and the constraints (AE) are globally uniquely solvable. 
	\end{defi}  A suitable hypothesis for having this assumption is, for example, that $Img(t,X)\subseteq ImA(t)$, for all $X\in \mathbb{R}^n$ and $t\in \left[0,T\right] $.
	
	We denote  by  $J(t,X)=A(t)+R(t)f_X'(t,X)$,   $t\in \left[0,T\right] $ and $X\in \mathbb{R}^n$ the Jacobian matrix  of the constraint  $(\ref{equa6})$ 
	where $R(t)$ is the projector matrix along $ImA(t)$ with $R(t)A(t)=0$,  $t\in \left[0,T\right] $ (see Proposition \ref{prop1}  and \cite{Renate2003} for more explanations).
	\begin{defi} 
		The Jacobian $J(\cdot,\cdot)$ possesses a globally bounded inverse if there exists a positive constant $N$ such that $\left\| J(t,X)^{-1}\right\| \leq N$,  $X\in \mathbb{R}^n$, and  $t\in \left[0,T\right] $.
	\end{defi}
	The first main result of this section is given in the following theorem.
	\begin{theo}\label{theo1} [Main result 1] 
		Assume  that the equation \eqref{equa1}  is an index-1 SDAE and the Jacobian matrix $J(\cdot,\cdot)$ is continuous and possesses a globally bounded inverse.
		
		Assume  also that the conditions \eqref{equa3} (with p=2) and \eqref{equa4} are satisfied and $\mathbb{E}\left\|X_0 \right\|^2<\infty$.
		Then there exists a unique process $X(\cdot)$  solution of \eqref{equa1}  such that  $X(\cdot) \in \mathcal{M}^2([0,T];\mathbb{R}^n)$. Moreover we have,
		\begin{equation}\label{equa14}
			\mathbb{E}\left( \sup_{t\in\left[0,T \right] }\left\|X(t) \right\|^2 \right) <\infty.
		\end{equation}
	\end{theo}
	\begin{prev}	For the proof, our approach is inspired by the one used in \cite{Renate2003}\footnote{Note that in \cite{Renate2003} the coefficients are global Lipschitz and $A$ is a constant singular matrix.}. More precisely,  we transform the SDAE \eqref{equa1} into the inherent regular SDE \eqref{equa10a} with local Lipschitz conditions. To tackle challenges, we use the It\^o formula to reduce the solution $X(\cdot)$ of SDAE \eqref{equa1} to the solution  $u(\cdot)$  of a new SDE defined by \eqref{equa10a}. Finally, we use some results from classical  SDE to finalize the proof.\\
		The proof is divided into four parts:\\
		
		\underline{\textbf{Part 1: Transformation of the SDAE \eqref{equa1} into the inherent regular SDE \eqref{equa10a}}}
		
		We start by showing that we can distinguish the differential and algebraic solution components using special projectors.  Recall that Proposition \ref{prop1} provides the existence of two projectors $Q(t)$ and $R(t)$ such that  $ImQ(t)=KerA(t)$, and $R(t)A(t)=0$, $t\in \left[0, T\right] $.
		Therefore, any process  $X(\cdot)$  can be written as: 
		\begin{align}\label{equa5}
			X(t)&=P(t)X(t)+Q(t)X(t)=u(t)+v(t),\\
			& \quad u(t)\in ImP(t), \quad v(t)\in ImQ(t) ,\quad t\in\left[0,T \right]\nonumber.
		\end{align}
		
		We continue by multiplying the equation (\ref{equa1})  by the matrix $R(\cdot)$. Recall that $R(t)g(t, X(t))=0$  ,  $t\in \left[0, T \right] $ (see Definition \ref{def1}). Moreover, we have $A(t)v(t)=0$, $t\in \left[0, T \right] $  using equation $\eqref{equa5}$. Consequently, we obtain the following constraints:
		\begin{equation}\label{equa6}
			A(t)v(t)+R(t)f(t,u(t)+v(t))=0, ~~t\in \left[0, T \right]. 
		\end{equation}
		Recall that by Theorem \ref{theo1} the Jacobian matrix of $(\ref{equa6})$ possesses a globally bounded inverse. Consequently, there exists an unique solution of equation (\ref{equa6}) denoted by $v(\cdot)=\hat{v}(\cdot,u(\cdot))$ (see  Theorem 7 in \cite{idczak2014global}).  Additionally, the equality $(\ref{equa5})$ becomes
		\begin{equation}\label{equa7}
			X(t)=u(t)+\hat{v}(t,u(t)),\quad  t\in [0,T].
		\end{equation}
		
		We continue by multiplying the equation $(\ref{equa1})$ by the projector matrix $I-R(\cdot)$ and  we have
		\begin{equation}\label{equa8}
			A(t)dX(t)=(I-R(t))f(t,X(t))dt+g(t,X(t))dW(t), ~~ t\in [0,T]	.
		\end{equation}
		Indeed, note that we use the fact that:  $$R(t)A(t)=0\text{ and } R(t)g(t,X(t))=0, ~t\in \left[0,T \right].$$
		Let $D(\cdot)$ be a non-singular matrix provided by Proposition \ref{prop2}. Then we have  $D(t)A(t)=P(t)$ and  $A^-(t)=D(t)(I-R(t))$ is the  pseudo inverse of a matrix A(t) for all $ t\in [0,T]$. Moreover $A^-(t)A(t)=P(t)$ and $A(t)A^-(t)=(I-R(t))$,
		for all $ t\in [0,T]$.\\ Multiplying $(\ref{equa8})$ by the matrix $A^-(\cdot)$   we  obtain:  
		\begin{equation}\label{equa10}
			P(t)dX(t)=A^-(t)f(t,X(t))dt+A^-(t)g(t,X(t))dW(t),\text{ for } t\in \left[0,T\right].
		\end{equation}
		Note that: 
		$A^-(t)(I-R(t))=A^-(t)A(t)A^-(t)=A^-(t)$ for all $t\in \left[0,T\right],$ 
		(see (ii) in Definition \ref{def2}).\\
		
		We continue by defining $k(t,X)=P(t)X$, $t\in \left[0,T \right]$ and $X\in \mathbb{R}^n$. Observe that $k(\cdot,X(\cdot))$ is a It\^o process wherever $X(\cdot)$  is a It\^o  process.\\  
		Using the fact that  $P(\cdot)$ is an continuously differentiable projector (see (A2.2) in the Assumption 2 ) and using also the It\^o formula (Proposition \ref{propito}) we obtain:
		\begin{align*}
			dk(t,X)&=P'(t)Xdt+P(t)dX,~t\in \left[0,T \right].
		\end{align*}
		Consequently,
		\begin{align}\label{equa9}
			P(t)dX&= d(P(t)X)-P'(t)Xdt,~t\in \left[0,T \right].
		\end{align}
		Combining the equations $\eqref{equa10}$ and $\eqref{equa9}$ we obtain:
		\begin{equation}\label{aa}
			d(P(t)X(t))=(P'(t)X(t)+A^-(t)f(t,X(t)))dt+A^-(t)g(t,X(t))dW(t),~t\in \left[0,T \right].
		\end{equation}
		We recall  that we denoted by  $P(t)X(t)= u(t)$ and $X(t)=u(t)+\hat{v}(t,u(t)), ~t\in  \left[ 0,T\right]$. Therefore, we can rewrite the equation \eqref{aa} as follows
		\begin{align}\label{equa10a}
			u(t)-u_0&=\int_{0}^{t}P'(s)\left( u(s)+\hat{v}(s,u(s))\right) +A^-(s)f(s,u(s)+\hat{v}(s,u(s)))ds\nonumber\\
			&+\int_{0}^{t}A^-(s)g(s,u(s)+\hat{v}(s,u(s)))dW(s), 
			~u_0=P(0)X(0),~t\in  \left[ 0,T\right].
		\end{align}
		Observe that $u(\cdot)$ is the new unknown variable in \eqref{equa10a}.
		
		The equation \eqref{equa10a} is called the inherent regular SDE (under P) associated to the SDAE \eqref{equa1}. Consequently solving \eqref{equa1} is equivalent with solving \eqref{equa10a} and \eqref{equa6}.	\\
		
		Let $\hat{f}(t,u)=P'(t)(u+\hat{v}(t,u))+A^-(t)f(t, u+\hat{v}(t,u))$ and  $\hat{g}(t,u)=A^-(t)g(t, u+\hat{v}(t,u))$,  $u\in \mathbb{R}^n$, $ t\in \left[ 0,T\right] $. Observe that $\hat{f}: \left[ 0,T\right]\times\mathbb{R}^n \to \mathbb{R}^n $, and $\hat{g}: \left[ 0,T\right]\times\mathbb{R}^n\to \mathbb{R}^{n\times m}$ are measurable functions.
		
		Furthermore, we will demonstrate that the functions $\hat{f}(\cdot, \cdot)$ and $\hat{g}(\cdot, \cdot)$ are locally Lipschitz continuous with respect to the variable  $u$ (as indicated in \eqref{equa4}) and satisfy the monotone condition \eqref{eqat}. Consequently, we can conclude that there exists a unique solution process $u(\cdot)$ for the equation \eqref{equa10a}, where $u\in \mathcal{M}^2(\left[0, T\right],\mathbb{R}^n)$, as stated in Theorem \ref{theoexi}.
		We would like to mention that the solution of $\eqref{equa10a}$ does not depend on the choice of the projector $P(\cdot)$.\\
		
		Let us prove that the functions  $\hat{f}(\cdot, \cdot)$ and $\hat{g}(\cdot, \cdot)$ are locally Lipschitz continuous $(\ref{equa4})$  and satisfy the monotone condition  $(\ref{eqat})$.\\
		\newpage
		\underline{\textbf{Part 2: Locally Lipschitz and the monotone conditions on $\hat{f}(\cdot, \cdot)$ and $\hat{g}(\cdot, \cdot)$}}
		\begin{itemize}
			\item Recall that the implicit function $\hat{v}:\left[0, T\right]\times\mathbb{R}^n\to \mathbb{R}^n$ solves the equation:
			$$h(t,u,v)=A(t)v+R(t)f(t,u+v)=0, ~u,~v\in \mathbb{R}^n, ~t\in\left[0, T\right].$$ Moreover, we observe that $\hat{v}(\cdot,\cdot )$     is continuously differentiable with respect to $t$ and to $u$  and globally Lipschitz with respect to $u$ because:\\
			
			The function $h:\left[0, T\right]\times\mathbb{R}^n\times\mathbb{R}^n\to \mathbb{R}^n$ is continuously differentiable  with respect to $v$ and $u$  and the Jacobian matrix $J(\cdot,\cdot)$  possesses a globally bounded inverse.
			
			By Theorem 7 in \cite{idczak2014global}, we have:
			\begin{align*}
				0&=h_u'(t,u,\hat{v}(t,u))=\hat{v}'_u(t,u)h_v'(t,u,\hat{v}(t,u))+h_u'(t,u,\hat{v}(t,u)), ~u\in \mathbb{R}^n, \\
				~t\in\left[0, T\right] &.\\ \text{ Consequently },&\\
				\hat{v}'_u(t,u)&=-\left( h_v'(t,u,\hat{v}(t,u))\right) ^{-1}h_u'(t,u,\hat{v}(t,u)),~u\in \mathbb{R}^n, ~t\in\left[0, T\right].
			\end{align*}
			Finally,
			\begin{align*}
				\hat{v}'_u(t,u)&=-(A(t)+R(t)f'_X(t,u+\hat{v}(t,u)))^{-1}R(t)f'_X(t,u+\hat{v}(t,u)),\\
				&=-J^{-1}(t,u+\hat{v}(t,u))R(t)f'_X(t,u+\hat{v}(t,u)),  ~u\in \mathbb{R}^n, ~t\in\left[0, T\right].
			\end{align*}
			
			Consequently,	let $I_n$ been  the identity matrix in $\mathbb{R}^n$.	Then we have: 
			\begin{align*}
				I_n=J^{-1}(t,u+\hat{v}(t,u))J(t,u+\hat{v}(t,u)))&=J^{-1}(t,u+\hat{v}(t,u))A(t)\\
				&+J^{-1}(t,u+\hat{v}(t,u))R(t)f'_X(t,u+\hat{v}(t,u)),\\
				u\in \mathbb{R}^n, ~t\in\left[0, T\right].	&
			\end{align*}
			then we have:
			\begin{align*}
				I_n-J^{-1}(t,u+\hat{v}(t,u))A(t)&=J^{-1}(t,u+\hat{v}(t,u))R(t)f'_X(t,u+\hat{v}(t,u))=\hat{v}'_u(t,u),\\
				u\in \mathbb{R}^n, ~t\in\left[0, T\right]& .
			\end{align*}
			Recall that $J^{-1}(\cdot, \cdot) $ is globally bounded.    Therefore, $\hat{v}'_u(\cdot,\cdot)$ is also bounded. 
			Finally  $\hat{v}(\cdot, \cdot)$ is globally Lipschitz continuous with the constant $L_{\hat{v}}=N \left\|A \right\| _{\infty}$+n.
			
			\item Recall that, $\hat{f}(t,u)=P'(t)(u+\hat{v}(t,u))+A^-(t)f(t,u+\hat{v}(t,u)),~t\in \left[0, T\right],~u\in \mathbb{R}^n$. Let $q>0  $. We have that: 
			\begin{align}\label{equa11}
				\left\|\hat{f}(t,u_1)-\hat{f}(t,u_2) \right\|  &= \left\|P'(t)(u_1+\hat{v}(t,u_1))+A^-(t)f(t,u_1+\hat{v}(t,u_1))\right.\nonumber\\
				& \left.-P'(t)(u_2+\hat{v}(t,u_2))	-A^-(t)f(t,u_2+\hat{v}(t,u_2)) \right\| \nonumber \\
				&\leq \left|P'(t) \right| _F (1+L_{\hat{v}})\left\| u_1-u_2\right\|+ L_{q}\left| A^-(t)\right|_F\nonumber\\
				&	\times\left\|(u_1+\hat{v}(t,u_1)-(u_2+\hat{v}(t,u_2)) \right\|   \nonumber \\
				&\leq  \left|P'(t) \right| _F (1+L_{\hat{v}})\left\| u_1-u_2\right\|+L_{q} \left| A^-(t)\right|_F\left[\left\|u_1-u_2\right\|\right.\nonumber\\
				&\left.+\left\| \hat{v}(t,u_1)-\hat{v}(t,u_2) \right\|   \right]    \nonumber\\
				&\leq \left\|P' \right\| _{\infty} (1+L_{\hat{v}})\left\| u_1-u_2\right\|+ L_{q}\left\| A^-\right\|_{\infty}\left[\left\|u_1-u_2\right\|\right.\nonumber\\
				&\left.+L_{\hat{v}}\left\|u_1+u_2\right\|\right] \nonumber\\
				&\leq \left\|P' \right\| _{\infty} (1+L_{\hat{v}})\left\| u_1-u_2\right\|+ L_{q}(1+L_{\hat{v}})\left\| A^-\right\|_{\infty}\left\|u_1-u_2\right\| \nonumber\\
				&\leq (1+L_{\hat{v}})\left(\left\|P' \right\| _{\infty}+L_q\left\| A^-\right\|_{\infty} \right) \left\| u_1-u_2\right\|\nonumber\\
				&\leq L_{\hat{f}_q}\left\|u_1-u_2\right\|,~t\in \left[0,T\right],
			\end{align} 
			for  $u_1,u_2\in \mathbb{R}^n $ with $\left\|u_1 \right\|\lor \left\|u_2 \right\|\leq q$,
			and  $ L_{\hat{f}_q}=(1+L_{\hat{v}})\left(\left\|P' \right\| _{\infty}+L_q\left\| A^-\right\|_{\infty} \right).$ Consequently $\hat{f}(\cdot,\cdot)$ is locally Lipschitz with respect to $u$.
			\item Recall that, $\hat{g}(t,u)=A^-(t)g(t,u+\hat{v}(t,u)),~ t\in \left[0,T\right],~u\in \mathbb{R}^n$.\\ Let $q>0  $. We have that:
			\begin{align}\label{equa11a}
				\left\|\hat{g}(t,u_1)-\hat{g}(t,u_2) \right\|  &= \left\|A^-(t)g(t,u_1+\hat{v}(t,u_1))-A^-(t)g(t,u_2+\hat{v}(t,u_2)) \right\| \nonumber \\
				&\leq L_{q}\left| A^-(t)\right|_F\left\|(u_1+\hat{v}(t,u_1)-(u_2+\hat{v}(t,u_2) \right\|   \nonumber \\
				&\leq L_{q} \left\| A^-\right\|_{\infty}\left[\left\|u_1-u_2\right\|+\left\| \hat{v}(t,u_1)-\hat{v}(t,u_2) \right\|   \right]    \nonumber\\
				&\leq L_{q}\left\| A^-\right\|_{\infty}\left[\left\|u_1-u_2\right\|+L_{\hat{v}}\left\|u_1+u_2\right\|\right] \nonumber\\
				&\leq L_{q}(1+L_{\hat{v}})\left\| A^-\right\|_{\infty}\left\|u_1-u_2\right\| \nonumber\\
				&\leq L_{\hat{g}_q}\left\|u_1-u_2\right\|,~t\in \left[0,T\right],
			\end{align} 
			for $u_1,u_2\in \mathbb{R}^n $ with $\left\|u_1 \right\|\lor \left\|u_2 \right\|\leq q $,
			and  $ L_{\hat{g}_q}=L_{q}(1+L_{\hat{v}})\left\| A^-\right\|_{\infty}.$\\ Consequently $\hat{g}(\cdot, \cdot)$ is locally Lipschitz with respect to $u$.
			\item Moreover, $\hat{f}(\cdot, \cdot)$ and $\hat{g}(\cdot, \cdot)$ satisfy the monotone condition \eqref{eqat} with respect to $u$. We have:
			
			\begin{align*}
				\left\langle u^T, \hat{f}(t,u)  \right\rangle +\frac{1}{2}\left|\hat{g}(t,u)  \right|^2_F&=\left\langle u^T, P'(t) (u+\hat{v}(t,u))+A^-(t)f(t,u+\hat{v}(t,u))\right\rangle\\
				&+\left| A^-(t)g(t,u+\hat{v}(t,u))\right|^2_F \\
				&\leq \left\langle u^T, P'(t) (u+\hat{v}(t,u))\right\rangle +\left\langle u^T, \right. \\
				&\left. A^-(t)f(t,u+\hat{v}(t,u))\right\rangle +\frac{1}{2}\left| A^-(t)g(t,u+\hat{v}(t,u))\right|^2_F,\\
				&~t\in \left[ 0,T\right], u\in \mathbb{R}^n.
			\end{align*}
			According to the Assumption 1 (\ref{equa3}) with p=2 and using the fact that the function $\hat{v}(\cdot,\cdot)$ is globally Lipschitz we have:
			\begin{align*}
				\left\langle u^T, \hat{f}(t,u)  \right\rangle +\frac{1}{2}\left|\hat{g}(t,u)  \right|^2_F	&\leq \left\| u\right\| ^2 +\left| P'(t)\right|^2_F \left\| u+\hat{v}(t,u)\right\| ^2 \\
				&+k(1+\left\|u+\hat{v}(t,u) \right\|^2  )\\
				&\leq \left\| u\right\| ^2 +\left| P'(t)\right|^2_F \left\| u+\hat{v}(t,u)-\hat{v}(t,0)+\hat{v}(t,0)\right\| ^2 \\
				&+k(1+\left\|u+\hat{v}(t,u)-\hat{v}(t,0)+\hat{v}(t,0) \right\|^2  )\\
				&\leq \left\| u\right\| ^2 +2\left| P'(t)\right|^2_F\left(  \left\|u\right\| ^2+2L_{\hat{v}}\left\| u\right\| ^2+2 \left\| \hat{v}(t,0)\right\| ^2 \right)\\
				&+k(1+2\left\|u\right\| ^2+4L_{\hat{v}}\left\| u\right\| ^2+4\left\| \hat{v}(t,0) \right\|^2  )\\
				&\leq \left\| u\right\| ^2 +2\left| P'(t)\right|^2_F\left(  \left\|u\right\| ^2(1+2L_{\hat{v}}) +2 \left\| \hat{v}(t,0)\right\| ^2 \right)\\
				&+k(1+\left\|u\right\| ^2(2++4L_{\hat{v}})+4\left\| \hat{v}(t,0) \right\|^2  )\\
				&\leq \left\| u\right\| ^2(1+2(1+2L_{\hat{v}})\left| P'(t)\right|^2_F) +4\left| P'(t)\right|^2_F \left\| \hat{v}(t,0)\right\| ^2 \\
				&+k+k\left\|u\right\| ^2(2+4L_{\hat{v}})+4k\left\| \hat{v}(t,0) \right\|^2  )\\
				&\leq \left\| u\right\| ^2(1+2(1+2L_{\hat{v}})(\left\| P'\right\|^2_{\infty}+k) \\
				&+k+\left\| \hat{v}(t,0) \right\|^2(4k+4\left\| P'\right\|^2_{\infty})  )\\
				&\leq k_1(1+\left\| u\right\| ^2),~t\in \left[ 0,T\right], u\in \mathbb{R}^n.
			\end{align*}
			Here $$k_1=max_{t\in \left[ 0,T\right]}\left((1+2(1+2L_{\hat{v}})(\left\| P'\right\|^2_{\infty}+k),k+\left\| \hat{v}(t,0) \right\|^2(4k+4\left\| P'\right\|^2_{\infty}\right).  $$
			Consequently $\hat{f}(\cdot, \cdot)$ and $\hat{g}(\cdot, \cdot)$ satisfy the monotone condition \eqref{eqat} with respect to $u$.
		\end{itemize}
		We can now conclude that equation \eqref{equa10a} has a unique solution $u(\cdot)$ in $\mathcal{M}^2(\left[0,T\right],\mathbb{R}^n )$ according to Theorem \ref{theoexi}.\\ Finally, we can prove that the solution $X(\cdot)$ is also in $\mathcal{M}^2 ( \left[0,T \right], \mathbb{R}^n)$.\\
		\underline{\textbf{Part 3: The  solution $X(\cdot)$ of \eqref{equa1} belongs to  $	\mathcal{M}^2 ( \left[0,T \right], \mathbb{R}^n)$ }}\\
		
		Indeed,  we have:
		\begin{align*}
			\left\|X(t) \right\|  &= \left\|u(t)+\hat{v}(t,u(t))\right\| \nonumber \\
			&= \left\| u(t)+\hat{v}(t,0)+\hat{v}(t,u(t))-\hat{v}(t,0) \right\|   \nonumber \\
			&\leq \left\| u(t)+\hat{v}(t,0)  \right\| + \left\| \hat{v}(t,u(t))-\hat{v}(t,0)  \right\|    \nonumber\\
			&\leq \left\| u(t)\right\|+\left\| \hat{v}(t,0)\right\| +L_{\hat{v}}\left\|u(t) \right\|,~t\in \left[0,T\right]   \nonumber\\
			\int_{0}^{T}\left\|X(t) \right\|^2dt&\leq2(1+L_{\hat{v}})^2	\int_{0}^{T}\left\|u(t) \right\|^2dt+2T\left\| \hat{v}(t,0)\right\|^2_{\infty}  \nonumber \\
			\mathbb{E}\int_{0}^{T}\left\| X(t)\right\|^2dt &\leq 2(1+L_{\hat{v}})^2\mathbb{E}	\int_{0}^{T}\left\|u(t) \right\|^2dt+2T\left\| \hat{v}(t,0)\right\|^2_{\infty}  \nonumber.
		\end{align*}
		Using the fact that 	$u\in \mathcal{M}^2 ( \left[0,T \right], \mathbb{R}^n)$  (see Theorem \ref{theoexi}) we have: 
		\begin{align}\label{equa12}
			\mathbb{E}\int_{0}^{T}\left\| X(t)\right\|^2dt	\leq2T\left\| \hat{v}(t,0)\right\|^2_{\infty}+C(T)
			< \infty, \quad \quad \quad \quad  t\in\left[ 0,T\right].
		\end{align}
		Here  $C(T)$  is a  constant depending on $T$ and $L_{\hat{v}}$.\\ Consequently, we can conclude that
		the solution $X(\cdot)$ of the equation \eqref{equa1} exists and belongs to $ 	\mathcal{M}^2 ( \left[0,T \right], \mathbb{R}^n)$.\\
		It remains to prove that \eqref{equa14} is satisfied. More precisely, we want to prove that\\
		
		\underline{\textbf{ Part 4: $	\mathbb{E}\left( \sup_{t\in\left[0,T \right] }\left\|X(t) \right\|^2 \right) <\infty$}}\\
		
		We use the fact that $u(\cdot)$ is the solution of SDE $\eqref{equa10a}$ and Theorem \ref{theoexi}. Therefore, we have $\mathbb{E}\left( \sup_{t\in\left[0,T \right] }\left\|u(t) \right\|^2\right)<\infty$ . \\
		Let us show that \eqref{equa14} is satisfied.\\
		Recall that: 
		\begin{align*}
			\left\|X(t) \right\|  &= \left\|u(t)+\hat{v}(t,u(t))\right\| \nonumber \\
			&= \left\| u(t)+\hat{v}(t,0)+\hat{v}(t,u(t)))-\hat{v}(t,0) \right\|  \\
			&\leq \left\| u(t)+\hat{v}(t,0)  \right\| + \left\| u(t)-\hat{v}(t,0)  \right\|  \\
			&\leq \left\| u(t)\right\|+\left\| \hat{v}(t,0)\right\| +L_{\hat{v}}\left\|u(t) \right\|,   t\in \left[0,T \right].
		\end{align*}
		Consequently, we obtain the following estimation:
		\begin{align*}
			\sup_{t\in\left[0,T \right] }\left\| X(t)\right\|^2 &\leq 2(1+L_{\hat{v}})^2\sup_{t\in\left[0,T \right] }\left\|u(t) \right\|^2+2\sup_{t\in\left[0,T \right] }\left\| \hat{v}(0,t)\right\|^2.  
		\end{align*}
		Therefore,
		\begin{align*}
			\mathbb{E}\left( 	\sup_{t\in\left[0,T \right] }\left\| X(t)\right\|^2\right)  &\leq 2(1+L_{\hat{v}})^2 	\mathbb{E}\left( \sup_{t\in\left[0,T \right] }\left\|u(t) \right\|^2\right) +2\sup_{t\in\left[0,T \right] }\left\| \hat{v}(0,t)\right\|^2.
		\end{align*}
		We can now apply Theorem \ref{theoexi1}   with $p=2$  and  $r_2=2$ and  we have:
		\begin{align*}
			\mathbb{E}\left( 	\sup_{t\in\left[0,T \right] }\left\| X(t)\right\|^2\right)	&< \infty, 
		\end{align*}
		which completes the proof of Theorem \ref{theo1}.
	\end{prev}
	
	We proceed with the following remark, which is essential in the proof of Theorem \ref{theo2} in the subsequent section.
	\begin{rmk}\label{rem4} 
		A consequence of the Assumption 1 (A1.1)   \eqref{equa3} is:
		\begin{align}\label{equa13b}
			\left\langle (P(t)X)^T,A^-(t)f(t,X)\right\rangle +\dfrac{p-1}{2}\left|A^-(t)g(t,X) \right|^2_F &\leq K(1+\left\|P(t)X \right\|^2 ),\\
			~t\in \left[ 0,T\right],~X\in \mathbb{R}^n \nonumber.
		\end{align}
	\end{rmk}
	Indeed, from inequality $\eqref{equa3}$ we have :
	\begin{align*}
		\left\langle (P(t)X)^T,A^-(t)f(t,X)\right\rangle +\dfrac{p-1}{2}\left|A^-(t)g(t,X) \right|^2_F &\leq k(1+\left\|X \right\|^2 )= k(1+\left\|u+\hat{v}(u,t) \right\|^2 )\\
		&\leq  k(1+\left\|u+\hat{v}(u,t)-v(0,t)+v(0,t)\right\|^2 )\\
		&\leq k(1+4\left\| \hat{v}(0,t)\right\|^2+2(2+L_{\hat{v}})^2\left\| u\right\|^2)\\
		&\leq K(1+\left\| P(t)X\right\|^2),~t\in \left[ 0,T\right],X\in \mathbb{R}^n,
	\end{align*}
	with $K=k\max\left(1+4\left\| \hat{v}(0,\cdot)\right\|^2_{\infty}, 2(2+L_{\hat{v}})^2\right).$
	
	We have all the necessary results for continuing to study the regularity of the solution $X(\cdot)$ of the equation \eqref{equa1}.

	\subsection{More results on the regularity of the solution for the non-autonomous SDAEs}
	
	In this section, we present regularity results for the solution of \eqref{equa1}.

	\begin{theo}\label{theo2} [Main result 2] \\
		Let $p \geq 2$ and $\mathbb{E}(\left\|X_0\right\|^p)<\infty$. Assume that Assumptions 1 and 2 hold. Then the solution $X(\cdot)$ of the equation $(\ref{equa1})$ satisfies the relation.
		$$\mathbb{E}(\left\|X(t)\right\|^p)\leq C(t),~ t\in \left[0,T\right], $$
		$C(\cdot)$ is the function of $t$.
	\end{theo}
	
	\begin{prev} In this proof we use the inequality: $(a+b)^n\leq 2^{n-1}(a^n+b^n)$. Here $a$ and $b$ are positive numbers. We use a similar approach as in \cite[ Theorem 2.4.1]{Mao2007}.
		Let $u(\cdot)$ be the solution of equation $\eqref{equa10a}$.\\ We define the function $Z(\cdot)$  by: $$Z(u)=(1+\left\|u \right\|^2 )^{\frac{p}{2}}, ~u\in \mathbb{R}^n.$$
		Observe that $Z(\cdot)$ is an continuously differentiable function. Therefore we apply  It\^o's formula ( Proposition \ref{propito}) and have the following formulation:
		\begin{align}\label{equa16}
			Z(u(t))  &=Z(u_0,0)+\int_{0}^{t}\frac{\partial Z}{\partial u}(u(s))\hat{f}(s,u(s))ds+\int_{0}^{t}\frac{\partial Z}{\partial u}(u(s))\hat{g}(s,u(s))dW_s\nonumber\\
			&+\frac{1}{2}\int_{0}^{t}\frac{\partial^2 Z}{\partial u^2}(u(s))\left| \hat{g}(s, u(s))\right|_F^2ds,~ u\in \mathbb{R}^n, ~t\in \left[0,T \right].
		\end{align}	
		We continue by observing that:
		$$\frac{\partial Z}{\partial u}(u)=pu^T \left( 1+\left\|u\right\|^2 \right)^{\frac{p-2}{p}},$$
		$$ \text{ and   } ~~\frac{\partial^2 Z}{\partial u^2}(u)= p(1+\left\|u\right\|^2 )^{\frac{p-2}{2}}+p(p-2)(u^T)^2\left( 1+\left\|u\right\|^2 \right)^{\frac{p-4}{2}}, ~ u\in \mathbb{R}^n.$$

		We define: $$D(t)=p\int_{0}^{t} \left( 1+\left\|u(s)\right\|^2 \right)^{\frac{p-2}{2}}u(s)^T\hat{g}(u(s))dW_s,~t\in \left[0,T\right].$$ 
		Using the previous expressions and the inequality: $$\left\| u\right\|^2\leq 1+\left\| u\right\| ^2,~ u\in \mathbb{R}^n,$$ the equation  (\ref{equa16}) becomes:
		
		\begin{align*}
			\left( 1+	\left\|u(t) \right\|^2\right)^{\frac{p}{2}}  &=\left( 1+	\left\|u_0 \right\|^2\right)^{\frac{p}{2}}+\int_{0}^{t}pu(s)^T \left( 1+\left\|u(s)\right\|^2 \right)^{\frac{p-2}{2}}\hat{f}(s,u(s))ds\nonumber\\
			&+\frac{1}{2}\int_{0}^{t} p(1+\left\|u(s) \right\|^2 )^{\frac{p-2}{2}}\left| \hat{g}(s,u(s))\right|_F^2ds\nonumber\\
			&+\frac{1}{2}\int_{0}^{t}p(p-2)\left\| u(s)\right\| ^2\left( 1+\left\|u(s)\right\|^2 \right)^{\frac{p-4}{2}}\left| \hat{g}(s,u_s)\right|_F^2ds+D(t) \\
			&\leq 	2^{\frac{p-2}{2}}(1+	\left\|u_0 \right\|^{p})+\int_{0}^{t}pu(s)^T \left( 1+\left\|u(s)\right\|^2 \right)^{\frac{p-2}{p}}\hat{f}(s, u(s))ds\nonumber\\
			&+\frac{1}{2}\int_{t_0}^{t}p(1+\left\|u(s) \right\|^2 )^{\frac{p-2}{2}}\left| \hat{g}(s,u(s))\right|_F^2ds\nonumber\\
			&+\frac{1}{2}\int_{0}^{t}p(p-2)\left( 1+\left\|u(s)\right\|^2 \right)^{\frac{p-2}{2}}\left| \hat{g}(s,u_s)\right|_F^2ds+D(t)
			,~t\in \left[0,T\right].
		\end{align*}
		\begin{align}
			\left( 1+	\left\|u(t) \right\|^2\right)^{\frac{p}{2}} 	&\leq 	2^{\frac{p-2}{2}}(1+	\left\|u_0 \right\|^{p})+p\int_{0}^{t}u(s)^T \left( 1+\left\|u(s)\right\|^2 \right)^{\frac{p-2}{2}}\hat{f}(s,u(s))ds\nonumber\\
			&	+\frac{1}{2}p(p-1)\int_{0}^{t} (1+\left\|u(s) \right\|^2 )^{\frac{p-2}{2}}\left| \hat{g}(s,u(s))\right|_F^2ds+D(t)  \nonumber\\
			&\leq 	2^{\frac{p-2}{2}}(1+	\left\|u_0 \right\|^{p})+D(t)\nonumber\\
			&+p\int_{0}^{t}\left[\left\langle  u(s)^T,\hat{f}(s,u(s))\right\rangle+\frac{p-1}{2}\left| \hat{g}(s,u(s))\right|_F^2\right]  \left( 1+\left\|u(s)\right\|^2 \right)^{\frac{p-2}{2}} ds\nonumber\\
			&\leq 2^{\frac{p-2}{2}}(1+\left\|u_0 \right\|^{p})+D(t)+\nonumber\\
			&p\int_{0}^{t}\left[\left\langle (P(s)X(s))^T,A^-(s)f(s,X)\right\rangle +\frac{p-1}{2}\left| A^-(s)g(s,X)\right|_F^2\right]\left( 1+\left\|u(s)\right\|^2 \right)^{\frac{p-2}{2}}\nonumber\\
			&+\langle u(s)^T,P'(s)(u(s)+\hat{v}(s,u(s))) \rangle\left( 1+\left\|u(s)\right\|^2 \right)^{\frac{p-2}{2}}ds,~t\in \left[0,T\right].
		\end{align}
		
		We continue by using Remark 2 and the fact that the function $\hat{v}(\cdot, \cdot)$ is Lipschitz in order to  obtain:
		
		\begin{align}
			\left( 1+	\left\|u(t) \right\|^2\right)^{\frac{p}{2}}  
			&\leq 	2^{\frac{p-2}{2}}(1+	\left\|u_0 \right\|^{p})+pK\int_{0}^{t}(1+\left\| u(s)\right\|^2 )\left( 1+\left\|u(s)\right\|^2 \right)^{\frac{p-2}{2}}ds\nonumber\\
			&+ \int_{t_0}^{t} \left\| u(s)\right\|\left| P'(s)\right| _F\left[ \left\| u(s)\right\| +	L_{\hat{v}}\left\| u(s)\right\|+\left\|\hat{v}(s,0) \right\|\right]\nonumber\\
			&\times \left( 1+\left\|u(s)\right\|^2 \right)^{\frac{p-2}{2}}ds+ D(t)\nonumber\\
			&\leq 	2^{\frac{p-2}{2}}(1+	\left\|u_0 \right\|^{p})+pK\int_{0}^{t}(1+\left\| u(s)\right\|^2 )\left( 1+\left\|u(s)\right\|^2 \right)^{\frac{p-2}{2}}ds\nonumber\\
			&+ \int_{0}^{t}\left[  (1+L_{\hat{v}})\left\| u(s)\right\|^2\left| P'(s)\right| _F\right.\nonumber\\
			&\left.+\left\| u(s)\right\|\left| P'(s)\right| _F\left\|\hat{v}(s,0) \right\|\right] \left( 1+\left\|u(s)\right\|^2 \right)^{\frac{p-2}{2}}ds+D(t)\nonumber\\
			&\leq 	2^{\frac{p-2}{2}}(1+	\left\|u_0 \right\|^{p})+pK\int_{0}^{t}\left( 1+\left\|u(s)\right\|^2 \right)^{\frac{p}{2}}ds+ D(t)\nonumber\\
			&+ \left( 1+L_{\hat{v}}+\left\|\hat{v}(t,0) \right\|_\infty\left\| P'\right\| _{\infty}\right)  \int_{0}^{t}\left( 1+\left\| u(s)\right\|^2\right) \left( 1+\left\|u(s)\right\|^2 \right)^{\frac{p-2}{2}}ds, ~~t\in \left[ 0,T\right].\nonumber
		\end{align}
		Consequently,	
		\begin{align}
			\left( 1+	\left\|u(t) \right\|^2\right)^{\frac{p}{2}}        &\leq 	2^{\frac{p-2}{2}}(1+	\left\|u_0 \right\|^{p})+\left( pK+1+L_{\hat{v}}+\left\|\hat{v}(t,0) \right\|_\infty\left\| P'\right\| _{\infty}\right)\nonumber\\
			&\times \int_{0}^{t}\left( 1+\left\|u(s)\right\|^2 \right)^{\frac{p}{2}}ds+ D(t),~~t\in \left[ 0,T\right].\nonumber
		\end{align}
		Taking the expectation we obtain:
		\begin{align}\label{equa17}
			\mathbb{E}\left[\left( 1+	\left\|u(t) \right\|^2\right)^{\frac{p}{2}} \right]               	&\leq 	2^{\frac{p-2}{2}}(1+	\mathbb{E}\left[\left\|u_0 \right\|^{p}\right] )+\left( pK+1+L_{\hat{v}}+\left\|\hat{v}(t,0) \right\|_\infty\left\| P'\right\| _{\infty}\right)\nonumber\\
			&\times\mathbb{E}\left[\int_{0}^{t}\left( 1+\left\|u(s)\right\|^2 \right)^{\frac{p}{2}} ds\right] + \mathbb{E}(D(t)),t\in \left[ 0,T\right].
		\end{align}
		Recall that:
		$$\mathbb{E}(D(t))=p\mathbb{E}\left[\int_{0}^{t} \left( 1+\left\|u(s)\right\|^2 \right)^{\frac{p-2}{2}}u(s)^T\hat{g}(u(s))dW_s\right] =0, ~t\in \left[0, T\right].$$
		Let $w=pK+1+L_{\hat{v}}+\left\|\hat{v}(t,0) \right\|_\infty\left\| P'\right\| _{\infty}$ and we have:
		\begin{align*}
			\mathbb{E}\left[\left( 1+	\left\|u(t) \right\|^2\right)^{\frac{p}{2}} \right]               	&\leq 	2^{\frac{p-2}{2}}(1+	\mathbb{E}\left[\left\|u_0 \right\|^{p}\right] )+w\mathbb{E}\left[\int_{0}^{t}\left( 1+\left\|u(s)\right\|^2 \right)^{\frac{p}{2}} ds\right] \nonumber\\
			&\leq 	2^{\frac{p-2}{2}}(1+	\mathbb{E}\left[\left\|u_0 \right\|^{p}\right] )+w_q\int_{0}^{t}\mathbb{E}\left[\left( 1+\left\|u(s)\right\|^2 \right)^{\frac{p}{2}} \right]ds\nonumber\\
			&\leq 	2^{\frac{p-2}{2}}(1+	\mathbb{E}\left[\left\|u_0 \right\|^{p}\right] )\exp(wt),t\in \left[0,T\right].
		\end{align*}
		Then we can  write that: 
		\begin{align}\label{equa18}
			\mathbb{E}\left[\left( 1+	\left\|u(t) \right\|^2\right)^{\frac{p}{2}} \right]               	&\leq 	2^{\frac{p-2}{2}}(1+	\mathbb{E}\left[\left\|u_0 \right\|^{p}\right] )\exp(wt)\nonumber\\
			\mathbb{E}\left[\left\|u(t) \right\|^p \right]               	&\leq 	2^{\frac{p-2}{2}}(1+	\mathbb{E}\left[\left\|u_0 \right\|^{p}\right] )\exp(wt),~t\in \left[0,T\right].
		\end{align}
		Finally, we have:
		\begin{align}\label{equa19}
			\mathbb{E}\left[\left\|X(t) \right\|^p \right]&=\mathbb{E}\left[\left\|u(t)+\hat{v}(u(t)) \right\|^p \right]\nonumber\\
			&=\mathbb{E}\left[\left\|u(t)+ \hat{v}(u(t))-\hat{v}(t,0)+\hat{v}(t,0) \right\|^p \right]\nonumber\\
			&\leq2^{p-1}\mathbb{E}\left[(1+L_{\hat{v}})^p\left\|u(t)\right\| ^p+ \left\| \hat{v}(t,0)\right\|^p \right]\nonumber\\
			&\leq2^{p-1}\left[ (1+L_{\hat{v}})^p\mathbb{E}\left[\left\|u(t)\right\| ^p\right]+ \left\| \hat{v}(t,0)\right\|^p \right] , ~t\in \left[0,T\right].\nonumber\\
			\text{ Using the  }& \text{ estimation \eqref{equa18} we obtain}:\nonumber\\
			\mathbb{E}\left[\left\|X(t) \right\|^p \right]&\leq2^{p-1}\left[ (1+L_{\hat{v}})^p\left[	2^{\frac{p-2}{2}}(1+	\mathbb{E}\left[\left\|u_0 \right\|^{p}\right] )\exp(wt)\right]+ \left\| \hat{v}(t,0)\right\|^p\right] ,~t\in \left[0,T\right].
		\end{align}
		Consequently  $	\mathbb{E}\left[\left\|X(t) \right\|^p\right]\leq C(t)$.\\
		Here, $C(t)=2^{p-1}\left[ (1+L_{\hat{v}})^p\left[	2^{\frac{p-2}{2}}(1+	\mathbb{E}\left[\left\|u_0 \right\|^{p}\right] )\exp(wt)\right]+ \left\| \hat{v}(t,0)\right\|^p\right] ,~t\in \left[0,T\right]$,\\
		and the proof is completed.
	\end{prev}\\
	
	We continue by providing a stronger estimation concerning the supremum overtime on the solution $X(\cdot)$ of the equation \eqref{equa1}.
	
	\begin{theo}\label{theo3}[Main result 3]\\
		Let $d\geq 2$   and $\mathbb{E}(\left\|X_0\right\|^d)<\infty$. Assume that Assumption 2
		hold. Then the solution of the equation \eqref{equa1} satisfies the following estimation:
		\begin{equation}\label{equa27}
			\mathbb{E}\left(\sup_{t\in\left[0,T \right] }\left\|X(t) \right\|^{d-r_2+2}  \right)\leq M  .
		\end{equation}
		Where $M$ is a constant depending on $X_0$,T,$r_2$, i.e., $M=M_{T,r_2,X_0}$ and $r_2$ is given by Assumption 2.
	\end{theo}
	For the proof of the previous theorem, we need several well-known lemmas:
	\begin{lem}\label{lem1} [Burkholder-Davis-Gundy inequality]\\
		Let $p>0$ and $g\in L^2(\mathbb{R}_+, \mathbb{R}^{n\times m}) $. Let
		$x(t)=\int_{0}^{t}g(s)dw(s)$ and  $B(t)=\int_{0}^{t}\left|g(s) \right|^2ds $, $t\in \left[0,T\right].$  Then there exist $c_p>0$ and $k_p>0$ such that:
		\begin{equation}
			c_p\mathbb{E}\left(B(t) \right)^{\frac{p}{2}}\leq \mathbb{E}\left(\sup_{s\in\left[0,T \right] } \left\| x(s)\right\|^p \right) \leq k_p\mathbb{E}\left(\left\|B(t) \right\|^{\frac{p}{2}}  \right),~ t\in \left[0,T \right]. 
		\end{equation}
		In particular, one may take:
		$$c_p=\left(\frac{p}{2} \right)^p,\quad k_p=\left(\frac{32}{p} \right)^{\frac{p}{2}} \quad if \quad 0<p<2, $$
		$$c_p=1,\quad k_p=4 \quad if \quad p=2, $$
		$$ \text{ and }~~c_p=\left(2p\right)^{\frac{-p}{2}},\quad k_p=\left(\frac{p^{p+1}}{2(p-1)^{p-1}} \right)^{\frac{p}{2}} \quad if \quad p>2. $$
		
	\end{lem}
	For proof of  Lemma \ref{lem1} see \cite[Theorem 7.3]{Mao2007} 
	\begin{lem}\label{lem2} Let $p \geq 2$, we have the following:
		\begin{equation}\label{eqau28}
			h^{p-2}v\leq\frac{p-2}{p}h^p+\frac{2}{p}v^{\frac{p}{2}}, \quad h,v \geq 0.
		\end{equation}
		
	\end{lem}
	\begin{prev}
		We start by recalling the Young inequality i.e.:
		\begin{equation*}
			ab\leq\frac{1}{r}a^r+\frac{1}{q}b^{q}, \quad a,b \geq 0,\quad and\quad \frac{1}{r}+\frac{1}{q}=1.
		\end{equation*}
		We conclude by taking:   $$a=h^{p-2}, ~b=v, ~r=\frac{p}{p-2}~ \text{ and }  ~q=\frac{p}{2}.$$\\ Observe that we have: 
		$\frac{1}{r}+\frac{1}{q}=\frac{p-2}{p}+\frac{2}{p}=1.$
	\end{prev}\\
	At this point we have the necessary results for continuing with the proof of Theorem \ref{theo3}.\\
	\begin{prev}  We use the inequality $(a+b)^n\leq 2^{n-1}(a^n+b^n)$ and the fact that the function $\hat{v}(\cdot, \cdot)$ given by equation \eqref{equa6} is Lipschitz with respect to $u$. Let $p=d-r_2+2$ and $X(\cdot)$ be the solution of the SDAE $\eqref{equa1}$ then we have: 	
		\begin{align}\label{equa24}
			\mathbb{E}\left[\sup_{s\in\left[0,T \right] }\left\|X(s) \right\|^{p} \right]&=\mathbb{E}\left[\sup_{s\in\left[0,T \right] }\left\|u(s)+\hat{v}(s,u(s)) \right\|^{p} \right]\nonumber\\
			&=\mathbb{E}\left[\sup_{s\in\left[0,T \right] }\left\|u(s)+ \hat{v}(u(s))-\hat{v}(s,0)+\hat{v}(s,0) \right\|^{p} \right]\nonumber\\
			&\leq2^{p-1}\mathbb{E}\left[\sup_{s\in\left[0,T \right] }\left[(1+L_{\hat{v}})^{p}\left\|u(s)\right\| ^{p}+ \left\| \hat{v}(s,0)\right\|^{p}\right]\right]\nonumber\\
			&\leq2^{p-1}\left[ (1+L_{\hat{v}})^{p}\mathbb{E}\left[\sup_{s\in\left[0,T \right] }\left\|u(s)\right\| ^{p}\right]+ \sup_{s\in\left[0,T \right] }\left\| \hat{v}(s,0)\right\|^{p} \right].
		\end{align}
		Let us show that $\mathbb{E}\left[\sup_{s\in\left[0,T \right] }\left\|u(s)\right\| ^{p}\right]<\infty$. For doing this we follow similar strategies as in \cite{mao2015truncated} Lemma 4.2\\
		We define  $Z(u)=\left\|u \right\|^{p}, ~u\in\mathbb{R}^n $. Observe that $Z(\cdot)$ is continuously differentiable  and we have:\\
		$$ \frac{\partial Z}{\partial u}=pu^T\left\|u \right\|^{p-2} \text{ and }$$
		$$\frac{\partial^2 Z}{\partial u^2}=p\left\|u \right\|^{p-2}+p(p-2) u^2\left\|u \right\|^{p-4}, u\in \mathbb{R}^n.$$ Using It\^o is  formula (Proposition \ref{propito}) and the fact that $u(\cdot)$ is a solution of \eqref{equa10a}  in order to obtain:
		\begin{align*}
			\left\|u(t) \right\|^{p}& =\int_{0}^{t}pu(s)^T\left\|u(s) \right\|^{p-2}\hat{f}(s,u(s))ds+ \int_{0}^{t}p\left\|u(s) \right\|^{p-2}u(s)^T\hat{g}(s,u(s))dW(s)\\
			&+\frac{1}{2}\int_{0}^{t }\left(p(p-2)\left\| u(s)\right\|^2\left\|u(s) \right\|^{p-4}+p\left\|u(s) \right\|^{p-2}\right) \left\| \hat{g}(s,u(s))\right\|^2ds+\left\| u_0\right\|^{p} \\
			&=\int_{0}^{t}p\left\|u(s) \right\|^{p-2}u(s)^T\hat{f}(s,u(s))ds+ \int_{0}^{t}p\left\|u(s) \right\|^{p-2}u(s)^T\hat{g}(s,u(s))dW(s)\\
			&+\frac{1}{2}\int_{0}^{t }\left(p(p-2)\left\|u(s) \right\|^{p-2}+p\left\|u(s) \right\|^{p-2}\right) \left\| \hat{g}(s,u(s))\right\|^2ds+\left\| u_0\right\|^{p} \\
			&=\left\| u_0\right\|^p+p\int_{0}^{t}\left( \langle u(s)^T,\hat{f}(s,u(s))\rangle+\frac{p-1}{2}\left\| \hat{g}(s,u(s))\right\|^2\right) \left\|u(s) \right\|^{p-2}ds\\
			&+ \int_{0}^{t}p\left\|u(s) \right\|^{p-2}u(s)^T\hat{g}(s,u(s))dW(s)\\ 
			&=\left\| u_0\right\|^{p} +p\int_{0}^{t}\left( \langle (P(s)X(s))^T,A^-(s)f(s,X(s))\rangle+\frac{p-1}{2}\left\|A^-(s)g(s,X(s))\right\|^2\right.\nonumber\\
			&\left.+\langle (P(s)X(s))^T,P'(s)X(s)\rangle\right)\left\|u(s) \right\|^{p-2}ds\\
			&+ \int_{0}^{t}p\left\|u(s) \right\|^{p-2}u(s)^T\hat{g}(s,u(s))dW(s),~~t\in \left[0,T\right].
		\end{align*}
		Using the inequality \eqref{equa3} we have:
		\begin{align*}
			\left\|u(t) \right\|^{p}&\leq 	\left\|u_0 \right\|^{p}+ pk\int_{0}^{t}\left\|u(s) \right\|^{p-2}\left( 1+	\left\|X(s) \right\|^2\right)ds\nonumber\\
			&+p\int_{0}^{t}\left| P(s)\right|_F\left| P'(s)\right|_F\left\| X(s)\right\|^2\left\|u(s)\right\|^{p-2}ds\nonumber\\ &+p\int_{0}^{t}\left\|u(s)\right\|^{p-2}u(s)^T\hat{g}(s,u(s))dW_s\\
			&\leq 	\left\|u_0 \right\|^{p}+ p\left( k+ \left\| P\right\|_{\infty}\left\| P'\right\|_{\infty} \right) \int_{0}^{t}\left\|u(s) \right\|^{p-2}\left( 1+	\left\|X(s) \right\|^2\right)ds\\ &+p\int_{0}^{t}\left\|u(s)\right\|^{p-2}u(s)^T\hat{g}(s,u(s))dW_s,~~t\in \left[0,T\right].
		\end{align*}
		We continue by Letting $w^1=p\left( k+\left\| P\right\|_{\infty}\left\| P'\right\|_{\infty}\right) $. Taking the supremum over $t$ in both side we obtain:
		\begin{align*}
			\sup_{t\in\left[0,T \right]}\left\|u(t) \right\|^{p} & \leq\left\|u_0 \right\|^{p}+ w^1\int_{0}^{T}\left\|u(s) \right\|^{p-2}\left( 1+	\left\|X(s) \right\|^2\right)ds\\
			&+p\sup_{t\in\left[0,T \right] }\left| \int_{0}^{t}\left\|u(s)\right\|^{p-2}u(s)^T\hat{g}(s,u(s))dw_s \right| .
		\end{align*}
		Taking the expectation, we obtain:
		\begin{align*}
			\mathbb{E}\left[  \sup_{t\in\left[0,T \right]}\left\|u(t) \right\|^{p} \right] & \leq\mathbb{E}\left\|u_0 \right\|^{p}+ w^1\mathbb{E}\left[ \int_{0}^{T}\left\|u(s) \right\|^{p-2}\left( 1+	\left\|X(s) \right\|^2\right)ds\right] \\
			&+p\mathbb{E}\left[\sup_{t\in\left[0,T \right] }\left|\int_{0}^{t}\left\|u(s)\right\|^{p-2}u(s)^T\hat{g}(s,u(s))dw_s \right| \right].
		\end{align*}
		We continue by using Lemma \ref{lem1} with $p=1$ and we have:
		\begin{align*}
			\mathbb{E}\left[  \sup_{t\in\left[0,T \right]}\left\|u(t) \right\|^{p} \right] &\leq \mathbb{E}\left\|u_0 \right\|^{p}+ w^1\mathbb{E}\left[ \int_{0}^{T}\left\|u(s) \right\|^{p-2}\left( 1+	\left\|X(s) \right\|^2\right)ds\right]   \\
			&+4\sqrt{2}p\mathbb{E}\left[\left|\int_{0}^{T}\left\|u(s)\right\|^{2p-2}\left| \hat{g}(s,u(s))\right|^2_Fds \right|^{\frac{1}{2}}\right] \\
			&\leq \mathbb{E}\left\|u_0 \right\|^{p}+ w^1\mathbb{E}\left[ \int_{0}^{T}\left\|u(s) \right\|^{p-2}\left( 1+	\left\|X(s) \right\|^2\right)ds\right]  \\
			&+\mathbb{E}\left[\left|\sup_{t\in\left[0,T \right]}\left\|u(t)\right\|^{p}32p^2 \int_{0}^{T}\left\|u(s)\right\|^{p-2}\left| \hat{g}(s,u(s))\right|^2_Fds \right|^{\frac{1}{2}}\right].
		\end{align*}
		We apply the inequality $\sqrt{ab}\leq \frac{1}{2}a+\frac{1}{2}b \text{ for } a \text{ , } b \geq 0$  and we obtain:
		\begin{align*}
			\mathbb{E}\left[  \sup_{t\in\left[0,T \right]}\left\|u(t) \right\|^{p} \right] 
			&\leq \mathbb{E}\left\|u_0 \right\|^{p}+ w^1\mathbb{E}\left[ \int_{0}^{T}\left\|u(s) \right\|^{p-2}\left( 1+	\left\|X(s) \right\|^2\right)ds\right]\\ &+\frac{1}{2}\mathbb{E}\left[\sup_{t\in\left[0,T \right]}\left\|u(t)\right\|^{p} )\right]+16p^2\mathbb{E}\left[ \int_{0}^{T}\left\|u(s)\right\|^{p-2}\left| \hat{g}(s,u(s))\right|^2_Fds \right]\\
			&\leq 2\mathbb{E}\left\|u_0 \right\|^{p}+ 2w^1\mathbb{E}\left[ \int_{0}^{T}\left\|u(s) \right\|^{p-2}\left( 1+	\left\|X(s) \right\|^2\right)ds\right]  \\
			&+32p^2\mathbb{E}\left[ \int_{0}^{T}\left\|u(s)\right\|^{p-2}\left| \hat{g}(s,u(s))\right|^2_Fds \right]\\
			&\leq 2\mathbb{E}\left\|u_0 \right\|^{p}+ 2w^1\mathbb{E}\left[ \int_{0}^{T}\left\|u(s) \right\|^{p-2}\left( 1+	\left\|X(s) \right\|^2\right)ds\right]  \\
			&+32p^2\left\|A^-\right\|_{\infty}^2c^2 \mathbb{E}\left[ \int_{0}^{T}\left\|u(s)\right\|^{p-2}(1+\left\|X(s)\right\|^{r_2})ds \right],~t\in \left[0,T\right].
		\end{align*}
		We continue by using Lemma \ref{lem2} to obtain:
		
		\begin{align*}
			\mathbb{E}\left[  \sup_{t\in\left[0,T \right]}\left\|u(t) \right\|^{p} \right] 
			&\leq 2\mathbb{E}\left\|u_0 \right\|^{p}+ 2w^1\mathbb{E}\left[ \int_{0}^{T}\left[  \frac{p-2}{p}\left\|u(s) \right\|^{p}+\frac{2}{p}\left( 1+	\left\|X(s) \right\|^2\right)^{\frac{p}{2}}\right] ds\right]   \\
			&+32p^2\left\|A^-\right\|_{\infty}^2c^2 \mathbb{E}\left[ \int_{0}^{T}\left[ \frac{p-1}{p}\left\|u(s)\right\|^{p}\right]ds\right.\nonumber\\ &\left.+ \int_{t_0}^{T}\left[\frac{2}{p}(1+\left\|X(s)\right\|^{r_2})^{\frac{p}{2}}\right]ds \right]	\\
			&\leq 2\mathbb{E}\left\|u_0 \right\|^{p}+ \frac{p-2}{p}(2w^1+32p^2\left\|A^-\right\|_{\infty}^2c^2)\mathbb{E}\left[ \int_{0}^{T}\left\|u(s) \right\|^{p}ds\right]   \\
			&+4\frac{w^1}{p}\mathbb{E}\int_{0}^{T}\left( 1+	\left\|X(s) \right\|^2\right)^{\frac{p}{2}}ds\nonumber\\
			&+32p^2\left\|A^-\right\|_{\infty}^2c^2 \mathbb{E}\left[ \int_{0}^{T}\left[\frac{2}{p}(1+\left\|X(s)\right\|^{r_2})^{\frac{p}{2}}\right]ds \right],~t\in \left[0,T\right].
		\end{align*}
		Combining the estimation  \eqref{equa18} and Theorem \ref{theo2}  and letting:\\
		\begin{eqnarray*}
			w_2&=&2\mathbb{E}\left\|u_0 \right\|^{p}+ \frac{p-2}{p}(2w^1+32p^2\left\| A^-\right\|_{\infty}^2c^2) \int_{0}^{T}\mathbb{E}\left[\left\|u(s) \right\|^{p}\right]ds +4\frac{w^1}{p}\int_{0}^{T}\mathbb{E}\left( 1+	\left\|X(t) \right\|^2\right)^{\frac{p}{2}}ds\\
			&&+64(2^{\frac{p-2}{2}})pc^2\left\|A\right\|_{\infty}^2T,
		\end{eqnarray*}
		we can write:
		\begin{align*}
			\mathbb{E}\left[  \sup_{t\in\left[0,T \right]}\left\|u(t) \right\|^{p} \right]	&\leq w_2+		64(2^{\frac{p-2}{2}})p\left\|A^-\right\|_{\infty}^2c^2 \left[ \int_{0}^{T}\mathbb{E}(1+\left\|X(s)\right\|^{\frac{pr_2}{2}})ds \right],~t\in \left[0,T\right].
		\end{align*}
		Consequently: 
		\begin{align*}
			\mathbb{E}\left[  \sup_{t\in\left[0,T \right]}\left\|u(t) \right\|^{p} \right]	&< \infty	\text{ a.s}.
		\end{align*}
		Moreover, the estimation \eqref{equa24} becomes:
		\begin{align*}
			\mathbb{E}\left[ \sup_{t\in\left[0,T \right]}	\left\|X(t) \right\|^{p}\right]& \leq2^{p-1}\left[ (1+L_{\hat{v}})^{p}\mathbb{E}\left[\sup_{s\in\left[0,T \right] }\left\|u(s)\right\| ^{p}\right]+ \sup_{s\in\left[0,T \right] }\left\| \hat{v}(s,0)\right\|^{p} \right]\\
			&  \leq 2^{p-1} \sup_{s\in\left[0,T \right] }\left\| \hat{v}(s,0)\right\|^{p} +2^{p-1} (1+L_{\hat{v}})^{p}\\
			&\times\left\lbrace  w_2+		64(2^{\frac{p-2}{2}})p\left\| A^-\right\|_{\infty}^2c^2 \mathbb{E}\left[ \int_{0}^{T}(1+\left\|X(s)\right\|^{pr_2})ds \right]\right\rbrace \\
			&\leq M_{p,T,r_2,X_0}.
		\end{align*}
		Here
		\begin{align*}
			M_{p,T,r_2,X_0}&= 2^{p-1}\sup_{s\in\left[0,T \right] }\left\| \hat{v}(s,0)\right\|^{p} + 2^{p-1}(1+L_{\hat{v}})^{p}\\
			&\times\left\lbrace  w_2+		64(2^{\frac{p-2}{2}})p\left\| A^-\right\|_{\infty}^2c^2 \mathbb{E}\left[ \int_{0}^{T}(1+\left\|X(s)\right\|^{pr_2})ds \right]\right\rbrace,
		\end{align*}
		Therefore the proof is completed.
	\end{prev}

	\section{Example} This section presents an example illustrating our results.
	Let us consider  the following SDAE of index 1 in $\mathbb{R}^2$
	\begin{equation}\label{solu1}
		\begin{pmatrix}
			0 & 0  \\
			t^2+1 & 0
		\end{pmatrix}dX=f(t,X(t))dt+g(t,X(t))dW(t),~~ X(0)=X_0, ~~t\in\left[0,T \right].
	\end{equation}
	Clearly, the matrix $A(t)=\begin{pmatrix}
		0 & 0  \\
		t^2+1 & 0 
	\end{pmatrix}$ is  a singular matrix for  $t\in \left[0,T \right] $.
	We have taken  
	\begin{align*}
		f: \left[0,T \right]\times\mathbb{R}^2 \times  &\to \mathbb{R}^2\\
		(t,X_1,X_2)&\mapsto (X_2, \frac{(-X_1-X_1^3)}{t^2+1})=(f_1(t,X), f_2(t,X)),
	\end{align*}
	and 
	\begin{align*}
		g: \left[0,T \right]\times\mathbb{R}^2 &\to \mathbb{R}^{2\times2}\\
		(t,X_1,X_2)&\mapsto \begin{pmatrix}
			0 & 0  \\
			X_2 &  X_1^2+2X_1
		\end{pmatrix}= \begin{pmatrix}
			0 & 0  \\
			g_1(t,X) & g_2(t,X)
		\end{pmatrix}.
	\end{align*}
	Let us prove that $f(\cdot,\cdot)$ and $g(\cdot,\cdot)$ satisfy  the  locally Lipschitz with respect the variable $X$.
	The function $f_2(\cdot, \cdot)$ and $g_2(\cdot, \cdot)$  are clearly locally Lipschitz with respect the variable $X$. Consequently, the functions $f(\cdot, \cdot)$ and $g(\cdot, \cdot)$  are locally Lipschitz with respect to the variable $X$.  	
	%
	The study of the SDAE \eqref{solu1} is divided in five steps.
	\begin{itemize}
		\item[Step 1:] Let us find the projectors Q(t) and  P(t), $t\in\left[0,T \right]$ (see Proposition \ref{prop2}).\\
		$KerA(t)=\left\lbrace (X_1(t), X_2(t)) \text{ such that } X_1(t)=0, t\in\left[0,T \right] \right\rbrace=ImQ(t),~t\in\left[0,T \right]. $\\
		Consequently, we can choose the projector  $Q(t)=\begin{pmatrix}
			0 & 0  \\
			0 & 1
		\end{pmatrix}$ for all $t\in\left[0,T \right]$.\\ Observe that $Q^2(t)=\begin{pmatrix}
			0 & 0  \\
			0 & 1 
		\end{pmatrix}\begin{pmatrix}
			0 & 0  \\
			0 & 1 
		\end{pmatrix}=\begin{pmatrix}
			0 & 0  \\
			0 & 1
		\end{pmatrix},~t\in\left[0,T \right]$.
		
		We recall that the projector $P(t)$ is defined by using the projector  $Q(t)$ as following: \\
		$P(t)=I_2-Q(t)= \begin{pmatrix}
			1 & 0  \\
			0 & 0 
		\end{pmatrix}, ~t\in\left[0,T \right] .$\\
		We will prove that : $A^-(t)=\begin{pmatrix}
			0& \frac{1}{t^2+1} \\
			0 & 0
		\end{pmatrix},  ~t\in\left[0,T \right]. $
		We can prove that the function $f(\cdot, \cdot)$ and $g(\cdot, \cdot)$ satisfy the monotone condition (A1.1) in the Assumption 1  with respect the variable $X$ i.e we prove that:
		$$\left\langle (P(t)X)^T,A(t)^-f(t,X)\right\rangle +\dfrac{1}{2}\left|A^-(t)g(t,X) \right|^2_F \leq k(1+\left\|X \right\|^2 ),~t\in \left[ 0, T \right], X\in \mathbb{R}^2.$$
		We have:
		\begin{align*}
			\left\langle PX, A^-(t)f(t,X) \right\rangle &= \left\langle  \begin{pmatrix}
				1 & 0  \\
				0 & 0 
			\end{pmatrix} \begin{pmatrix}
				X_1  \\
				X_2 
			\end{pmatrix},\begin{pmatrix}
				0& \frac{1}{t^2+1} \\
				0 & 0
			\end{pmatrix}\begin{pmatrix}
				X_2  \\
				- \frac{(X_1+X_1^3)}{t^2+1}
			\end{pmatrix} \right\rangle \\
			&= \left\langle   \begin{pmatrix}
				X_1  \\
				0
			\end{pmatrix},  \begin{pmatrix}
				- \frac{(X_1+X_1^3)}{(t^2+1)^2}  \\
				0
			\end{pmatrix} \right\rangle \\
			&=-\frac{(\left|X_1 \right|^2+\left|X_1 \right|^4 )}{(t^2+1)^2} \\
			& =-\frac{\left|X_1 \right|^2 }{(t^2+1)^2} -\frac{\left|X_1 \right|^4 }{(t^2+1)^2} ,~t\in \left[ 0, T \right], X\in \mathbb{R}^2,
		\end{align*} and 
		\begin{align*}
			\frac{1}{2} \left|A^-g(t,X) \right|_F^2&= \frac{1}{2} \left|\begin{pmatrix}
				0& \frac{1}{t^2+1} \\
				0 & 0
			\end{pmatrix}\begin{pmatrix}
				0 & 0  \\
				X_2 &  X_1^2+2X_1
			\end{pmatrix}  \right|_F ^2	\\
			&=\frac{1}{2} \left|\begin{pmatrix}
				\frac{X_2}{t^2+1}  & \frac{X_1^2+2X_1}{t^2+1}  \\
				0&0  
			\end{pmatrix}  \right| _F^2\\
			& \leq \frac{\left| X_2\right| ^2}{(t^2+1)^2}+ \frac{\left| X_1\right| ^4}{(t^2+1)^2}+\frac{4\left| X_1\right| ^2}{(t^2+1)^2} ,~t\in \left[ 0, T \right], X\in \mathbb{R}^2.
		\end{align*}
		Finally we have:
		\begin{align*}
			\left\langle P(t)X, A^-(t)f(t,X) \right\rangle+\frac{1}{2} \left|A^-g(t,X) \right|_F^2&\leq\frac{3\left| X_1\right| ^2}{(t^2+1)^2}+\frac{\left| X_2\right| ^2}{(t^2+1)^2}\\
			&\leq\frac{3\left| X_1\right| ^2}{t^2+1}+\frac{3\left| X_2\right| ^2}{t^2+1}+\frac{3}{t^2+1}\\
			&\leq 3(1+\left| X_1\right| ^2+\left| X_2\right| ^2),~t\in \left[ 0, T \right], X\in \mathbb{R}^2.\\
		\end{align*} 
		And the monotone condition is satisfied.
		
		\item[Step 2:]  Let us find the projector $R(t)$, $t\in\left[0,T \right]$ and the constraints.\\
		The projector $R(\cdot)$ satisfies the relations $R(t)A(t)=0$ and $R^2(t)=R(t),~t\in \left[0,T \right] .$ Therefore we can choose $R(t)=\begin{pmatrix}
			1 & 0 \\
			0 & 0
		\end{pmatrix}, ~t\in\left[0,T \right]$.\\
		The constraint is given by the equation  $A(t)v+R(t)f(t,u+v)=0$, where $u\in ImP(t)$ and $v\in ImQ(t),  ~t\in\left[0,T \right]$. We replace $R(t$) and $A(t)$ and we obtain the constraints:
		\begin{align}\label{cons}
			\begin{pmatrix}
				0 & 0  \\
				t^2+1 & 0 
			\end{pmatrix}v(t)+\begin{pmatrix}
				1 & 0 \\
				0 & 0
			\end{pmatrix}f(t,u(t)+v(t))=0,~t\in\left[0,T \right].
		\end{align}
		
		\item [Step 3: ] We can show that the hypothesis of Theorem \ref{theo1} is satisfied, i.e., we want to show the Jacobian matrix $J(\cdot,\cdot)$ possesses a globally bounded inverse.
		
		\begin{align*}
			J(t,X)&=A(t)+R(t)f'_X(t,X)\\
			&=\begin{pmatrix}
				0 & 0  \\
				t^2+1 & 0 
			\end{pmatrix}+\begin{pmatrix}
				1 & 0  \\
				0 & 0
			\end{pmatrix}\begin{pmatrix}
				0 & 1  \\
				\frac{-(1+3X^2_1)}{t^2+1} & 0
			\end{pmatrix}\\
			&=\begin{pmatrix}
				0 & 0  \\
				t^2+1 & 0 
			\end{pmatrix}+\begin{pmatrix}
				0 & 1  \\
				0 & 0
			\end{pmatrix}\\
			&=\begin{pmatrix}
				0 & 1 \\
				t^2+1 & 0
			\end{pmatrix}, X\in \mathbb{R}^2,~t\in \left[0,T \right].
		\end{align*}
		$$\text{det}(J(t,X))= 		-(t^2+1)\neq 0 ,~X\in \mathbb{R}^2,~t\in \left[0,T \right].$$
		Consequently, there exists an implicit function $v(\cdot)=\hat{v}(\cdot,\cdot)$ solution of $\eqref{cons}$.\\ 
		Observe that 
		\begin{align}
			J^{-1}(t,X)=\begin{pmatrix}
				0 & \frac{1}{t^2+1} \\
				1 & 	0
			\end{pmatrix};~~X\in \mathbb{R}^2,~t\in \left[0,T \right].
		\end{align}
		Finally we have:
		\begin{align*}
			\left\| J^{-1}(t,X)\right\|^2&\leq 2,~X\in \mathbb{R}^2,~t\in \left[0,T \right].
		\end{align*}
		and then $J^{-1}(\cdot, \cdot)$ is globally bounded.\\
		Note that we use the fact that: 
		$$\frac{1}{t^2+1}\leq 1,~X\in \mathbb{R}^2,~t\in \left[0,T \right].$$
		So we have from \eqref{cons}:
		
		\begin{align}\label{cons1}
			\begin{pmatrix}
				0 & 0  \\
				t^2+1 & 0 
			\end{pmatrix}\begin{pmatrix}
				v_1(t) \\
				v_2(t)
			\end{pmatrix}+\begin{pmatrix}
				1 & 0 \\
				0 & 0
			\end{pmatrix}\begin{pmatrix}
				u_2(t)+v_2(t) \\
				\frac{-(u_1(t)+v_1(t))-(u_1(t)+v_1(t))^3}{1+t^2}
			\end{pmatrix}=0,~t\in\left[0,T \right].
		\end{align}
		From $\eqref{cons1}$ we have,  $v_2(t)=-u_2(t)$ and $v_1(t)=0, t\in \left[0,T \right].$ \\
		So we have $\hat{v}(t,u(t))=\begin{pmatrix}
			0 \\
			-u_2(t)
		\end{pmatrix}.$
		
		\item [Step 4: ]  Let us find the non-singular matrix  $D(t)$ and the pseudo-inverse matrix $A^-(t), t\in\left[0,T \right]$.\\
		Recall that $D(t)$ satisfies the relation $D(t)A(t)=P(t), t\in\left[0,T \right]$ and it is non singular matrix. We use the proof of Proposition \ref{prop2} to construct the matrix D. Consequently we obtain:\\ $D(t)=\begin{pmatrix}
			1 & \frac{1}{t^2+1} \\
			t^2+1 & 0
		\end{pmatrix}$  and\\  $A^-(t)=D(t)(I_2-R(t))=\begin{pmatrix}
			1 & \frac{1}{t^2+1}  \\
			t^2+1 & 0
		\end{pmatrix}\left( \begin{pmatrix}
			1 & 0  \\
			0 & 1 
		\end{pmatrix}-\begin{pmatrix}
			1 & 0  \\
			0 & 0
		\end{pmatrix}\right)=\begin{pmatrix}
			0 & \frac{1}{t^2+1} \\
			0 & 0
		\end{pmatrix}$,  \\ $t\in\left[0,T \right]. $\\
		Observe that , $A^-(t)A(t)=\begin{pmatrix}
			0& \frac{1}{t^2+1} \\
			0 & 0
		\end{pmatrix} \begin{pmatrix}
			0& 0  \\
			t^2+1 & 0
		\end{pmatrix} =\begin{pmatrix}
			1 & 0 \\
			0 &0
		\end{pmatrix} =P(t)$ and  \\
		$A(t)A^-(t)=\begin{pmatrix}
			0& 0  \\
			t^2+1 & 0
		\end{pmatrix}\begin{pmatrix}
			0 & \frac{1}{t^2+1} \\
			0 & 0
		\end{pmatrix} =\begin{pmatrix}
			0 & 0  \\
			0 & 1
		\end{pmatrix}=\begin{pmatrix}
			1 & 0  \\
			0 & 1
		\end{pmatrix}-\begin{pmatrix}
			1 & 0  \\
			0 & 0
		\end{pmatrix}=I-R(t), ~t\in\left[0,T \right].   $\\
		Consequently $A^-(\cdot)$ is the unique pseudo-inverse of the matrix $A(\cdot)$.
		\item [Step 5: ] Finally, we have the following associated SDE:\\
		\begin{align}\label{solu}
			u_1(t)-u_1(0)&=\int_0^t
			\frac{f_2(s,u(s)+\hat{v}(s,u(s)))}{t^2+1} ds\nonumber\\
			&+\int_0^t
			\frac{g_1(s,u(s)+\hat{v}(s,u(s)))}{t^2+1}dW_1(s)+\int_0^t
			\frac{g_2(s,u(s)+\hat{v}(s,u(s)))}{t^2+1}dW_2(s),~t\in\left[0,T \right].\nonumber\\
			u_2(t)-u_2(0)&=0. 
		\end{align}
		Observe that  $\begin{pmatrix}
			u_1(0)  \\
			u_2(0)
		\end{pmatrix}=P(0)X(0)=\begin{pmatrix}
			1 & 0 \\
			0& 0
		\end{pmatrix}\begin{pmatrix}
			X_1(0) \\
			X_2(0)
		\end{pmatrix}=\begin{pmatrix}
			X_1(0)  \\
			0
		\end{pmatrix}.$\\
		So $u_2(t)=0, ~t\in\left[0,T \right]$ and then $\hat{v}(t,u(t))=\begin{pmatrix}
			0 \\
			0
		\end{pmatrix}, ~t\in\left[0,T \right].$ \\
		Finally, we have:
		\begin{align}\label{solu}
			u_1(t)-u_1(0)&=\int_0^t
			\frac{f_2(s,u_1(s),0)}{t^2+1} ds+\int_0^t
			\frac{g_1(s,u_1(s),0)}{t^2+1}dW_1(s)+\int_0^t
			\frac{g_2(s,u_1(s),0)}{t^2+1}dW_2(s)\nonumber\\
			u_1(0)&=X_1(0),~t\in\left[0,T \right]. 
		\end{align}
		
	\end{itemize}

	Note that the equation \eqref{solu} is an SDE with local Lipschitz and monotone conditions. Consequently, the solution $u(\cdot)$ and respectively the solution $X(\cdot)$  of equations \eqref{solu} and respectively \eqref{solu1} exist and belong to $\mathcal{M}^2(\left[0,T \right], \mathbb{R}^2 )$.

	\section{Conclusion and Perspectives}
	
	The main objective of this paper was to establish the well-posedness and regularity of the solution for nonlinear non-autonomous stochastic differential algebraic equations (SDAEs) under local Lipschitz conditions. To achieve this, we employed a transformation technique to convert the initial SDAE into an ordinary stochastic differential equation (SDE) with algebraic constraints. In our case, the matrix $A(\cdot)$ was dependent on time $t$, which added complexity to the transformation process. However, we successfully managed to convert the SDAE \eqref{equa1} into an SDE \eqref{equa10a} with algebraic constraints \eqref{equa6}. This was accomplished through the use of matrix transformations, It\^o processes, and the implicit function theorem. We demonstrated that the resulting equation, known as the inherent regular SDE, satisfied the local Lipschitz and monotone conditions for both the drift and diffusion coefficients. As a result, the solution to the SDAE exists and is unique. Additionally, we examined the regularity of the solution for SDAEs under certain conditions or assumptions, aiming to achieve a higher order of convergence in numerical simulations.
	
	In most applications, finding analytic solutions for SDAEs is impossible or not easy. Therefore we plan to continue the study of SDAE of index-1 by providing numerical methods for computing the approximated solutions.   
	\section{Appendix}
	\appendix
	\section{Proof of Proposition \ref{prop2}}\label{Appendix}
	
	\begin{prev}
		We know that $\mathbb{R}^n=kerB\bigoplus ImB$. Let us assume that $\left\lbrace e_i\right\rbrace_{i=1}^{n} $  is a basis of $\mathbb{R}^n$, such that $\left\lbrace e_i\right\rbrace_{i=1}^{q} $ is a basis of $KerB$ and of $\left\lbrace e_i\right\rbrace_{i=q+1}^{n} $ is  a basis of $ImB$. Note that $q$ is the dimension of $KerB$ or $ImQ$. Consequently,  for all $e_i$, $i=1,..., n$ we have:
		\begin{equation*}
			\begin{cases}
				Be_i=0 & \text{if } e_i\in  KerB, \\
				Be_i\neq 0  & \text{if } e_i\in  ImB.
			\end{cases}
		\end{equation*}
		Therefore,  the basis $ \left\lbrace Be_i \right\rbrace_{i=q+1}^{n} $ is also a basis of $ImB$. By definition for any $w\in ImB$ there exists  $\left\lbrace \alpha_i\right\rbrace_{i=q+1}^{n} $ such that $w=\sum_{i=q+1}^{n} \alpha_iBe_i$. Remember that $P=I-Q$ and $QB=0$. In addition, we obtain the equality:
		\begin{equation}\label{y}
			Pw=\sum_{i=q+1}^{n}P \alpha_iBe_i=\sum_{i=q+1}^{n} \alpha_i(I-Q)Be_i=\sum_{i=q+1}^{n} \alpha_iBe_i=w. 
		\end{equation}
		Recall that $KerB=span\left\lbrace e_i\right\rbrace_{i=1}^{q}$, and $ImB=span\left\lbrace e_i\right\rbrace_{i=q+1}^{n}=span\left\lbrace Be_i \right\rbrace_{i=q+1}^{n}$. We continue by defining the operator L  such that:
		\begin{equation}\label{L}
			\left\{
			\begin{array}{ccc}
				Lz=z  &   \text{ if} ~~ z\in KerB,\\
				Lw=Bw  &   \text{ if} ~~ w\in ImB.
			\end{array} \right.
		\end{equation}

		Moreover, by definition,  for an  $x\in KerB\bigoplus ImB$  there exists $z\in KerB$ and $w\in ImB$ such that $x=z+w$. By using the equation \eqref{L} we have that $Lx=z+Bw$. \\
		Note that,  the operator :$$L: \mathbb{R}^n\to \mathbb{R}^n $$ is an one to one or an bijective  operator. Finally, for any $y\in \mathbb{R}^n$ there exists $z\in KerB$  and $w \in ImB $ such that $y=z+w$. In particular for $z=0$, we use the equation \eqref{y} 
		in other to obtain: $$L^{-1}By=L^{-1}B(0+w)=L^{-1}Bw=L^{-1}Lw=w=Pw=P(0+w)=Py \implies L^{-1}B=P. $$
		
		We conclude by choosing $D=L^{-1}$
	\end{prev}\\
	
	\textbf{Competing Interests:}
	The authors declare that they have no competing interests.
	\section{Declarations}

\textbf{Ethical Approval } 
(applicable for both human and/ or animal studies. Ethical committees, Internal Review Boards and guidelines followed must be named. When applicable, additional headings with statements on consent to participate and consent to publish are also required)  
 \emph{Not applicable}\\
\textbf{Funding }
(details of any funding received)
 \emph{Not applicable} \\
\textbf{Availability of data and materials  }
(a statement on how any datasets used can be accessed)
\emph{Not applicable} \\

%
%
	
\end{document}